\documentclass{amsart}

\usepackage{begnac}

\begin{document}

\title{Reconstruction of non-$\aleph_0$-categorical theories}

\author{Itaï \textsc{Ben Yaacov}}

\address{Itaï \textsc{Ben Yaacov} \\
  Université Claude Bernard -- Lyon 1 \\
  Institut Camille Jordan, CNRS UMR 5208 \\
  43 boulevard du 11 novembre 1918 \\
  69622 Villeurbanne Cedex \\
  France}

\urladdr{\url{http://math.univ-lyon1.fr/~begnac/}}

\thanks{Author supported by ANR projects GruPoLoCo (ANR-11-JS01-008) and AGRUME (ANR-17-CE40-0026).}

\svnId $Id: Groupoid.tex 4550 2021-02-02 10:56:30Z begnac $
\thanks{\textit{Revision} {\svnRevision} \textit{of} \svnDate}

\keywords{complete theory, groupoid, reconstruction}
\subjclass[2020]{03C15, 03C95, 03C30}

\begin{abstract}
  We generalise the correspondence between $\aleph_0$-categorical theories and their automorphism groups to arbitrary complete theories in classical logic, and to some theories (including, in particular, all $\aleph_0$-categorical ones) in continuous logic.
\end{abstract}

\maketitle

\setcounter{tocdepth}{1}
\tableofcontents

\section*{Introduction}

To every $\aleph_0$-categorical theory $T$ (all theories under consideration are in a countable language) one can associate the automorphism groups $G(T)$ of its unique countable model, equipped with the Polish group topology of simple convergence.
It is by now a classical result that $G(T)$ is a classifying invariant for the bi-interpretation class of $T$ (Ahlbrandt and Ziegler \cite{Ahlbrandt-Ziegler:QuasiFinitelyAxiomatisable}, but due to Coquand).
In more explicit terms, if $T$ and $T'$ are $\aleph_0$-categorical, then $G(T) \cong G(T')$ as topological groups if and only if there exists a bi-interpretation between $T$ and $T'$.
The same was later extended by Kaïchouh and the author in \cite{BenYaacov-Kaichouh:Reconstruction} to $\aleph_0$-categorical theories in continuous logic, where $G(T)$ is the automorphism group of the unique separable model.
These correspondences opened the door to many interactions between model theory and topological dynamics, with model-theoretic properties of $T$ corresponding to well-studied dynamical properties of $G(T)$, see for example \cite{BenYaacov-Tsankov:WAP,Ibarlucia:DynamicalHierarchy,Ibarlucia:AutRand,BenYaacov-Ibarlucia-Tsankov:Eberlein}.

In the present paper, we propose to extend the correspondence between bi-interpretation classes of theories and topological groups (or group-like objects) beyond the $\aleph_0$-categorical realm.
One motivation for doing this comes from a desire to imitate the elegance of the original correspondence result.
Other motivations arise from applications of the correspondence between model-theoretic properties of $T$ and dynamical properties of $G(T)$.
When $T$ is not $\aleph_0$-categorical, one can no longer speak of ``the'' automorphism group of $T$.
Model-theoretic properties of $T$ still correspond to dynamical properties of all actions of automorphism groups of countable/separable models of $T$ on formulas, but the resulting criteria are far from being as elegant, or as useful (depending on context), as in the $\aleph_0$-categorical case.
Specifically, one needs to consider automorphism groups of all models of $T$ (or of sufficiently rich ones), and to know which functions on the group(s) correspond to formulas.

Let us make this a little more concrete using our favourite motivating example.
It is proved in \cite{BenYaacov:RandomVC} that the randomisation of a NIP theory is again NIP\@.
The proof much more analytic than model-theoretic, and there have since been several attempts to replace it with a different argument.
The only successful one, as far as we are aware, is by Ibarlucía \cite{Ibarlucia:AutRand}.
It only applies to $\aleph_0$-categorical $T$, and is based on the characterisation of NIP in terms of the representability in Rosenthal Banach spaces of a dynamical system associated with $G(T)$.
When $T$ is not $\aleph_0$-categorical, then, as in the previous paragraph, there still is a correspondence between NIP and Rosenthal representability of all actions of automorphism groups on formulas.
However, the one-by-one consideration of countable/separable models of $T$ does not pass well to the randomisation -- we can construct some separable models of $T^R$, but not all of them, as randomisations of models of $T$ -- so the criterion does not seem to be applicable.
The approach of the present paper allows us to consider all countable/separable models of $T$ jointly (rather than severally).
Recent results of Jorge Muñoz assert that this does commute quite well with randomisation, allowing us to hope to extend Ibarlucía's results.

\medskip

We achieve the desired correspondence by replacing topological groups with topological groupoids, which are briefly discussed in \autoref{sec:TopologicalGroupoid}.

We treat theories in classical and in continuous logic separately.
For a complete classical theory $T$ we construct in \autoref{sec:GroupoidTheoryClassical} a topological groupoid $\bG(T)$ over the Cantor space.
Roughly speaking, points in the base (in the Cantor space) are types of (codes for) models, and groupoid elements code isomorphisms between models of the source and target types.
We prove that $\bG(T)$ only depends on $T$ up to bi-interpretation, and conversely, in \autoref{sec:ClassicalReconstruction} we reconstruct $T$ up to bi-interpretation from $\bG(T)$.
It follows that
If $T$ is $\aleph_0$-categorical, then $\bG(T) \cong 2^\bN \times G(T) \times 2^\bN$, so our correspondence is a generalisation of the $\aleph_0$-categorical case.

The treatment of the continuous case is not quite as satisfactory.
In \autoref{sec:UniversalSkolemSort} we identify a sort of ``codes for models'' as a \emph{universal Skolem sort}.
While we do not know that one exists in full generality, we do know that:
\begin{itemize}
\item If it exists, then it is unique (up do a definable bijection).
\item All classical theories admit such a sort (constructed in \autoref{sec:GroupoidTheoryClassical}, motivating the general definition).
\item All $\aleph_0$-categorical theories (classical or continuous) admit such a sort.
\item If $T$ admits such a sort, then so does its randomisation $T^R$ (this is due to Jorge Muñoz, and is not proved here).
\end{itemize}
In \autoref{sec:GroupoidTheory}, assuming $T$ admits a universal Skolem sort, we construct $\bG(T)$ and reconstruct $T$ up to bi-interpretation.

This leaves quite a few open questions, which we present in \autoref{sec:Questions}.

\medskip

We should point out that in the context of categorical logic there exist results which also code a theory with a topological groupoid, in a very different fashion.
These include explicit constructions, such as Awodey and Forssell \cite{Awodey-Forssell:Duality}, as well as general ``there exists a groupoid that codes a topos that codes something'' arguments.
Awodey and Forssell consider models over subsets of a fixed uncountable set, so their groupoid is non-separable $T_0$ (but not $T_1$, since the closure of a singleton representing one model consists of all its sub-models).
Our construction, in contrast, yields a Polish groupoid (separable and completely metrisable as a topological space), and while we do not discuss this here, elementary embeddings of models of $T$ arise very differently, as the left-completion of the said groupoid.
To the extremely limited extent that we understand the more general constructions (the reader will forgive the author for his terrible lack of familiarity with categorical logic), similar differences apply there as well.

\section{Topological groupoids}
\label{sec:TopologicalGroupoid}

Let us recall the definition of a groupoid.
The definition is essentially equivalent to the one found in, say, Mackenzie \cite{Mackenzie:LieGroupoids}, except that we consider the base as a subset of the groupoid rather than as a separate space.

\begin{dfn}
  \label{dfn:Groupoid}
  A \emph{groupoid} is a set $\bG$ equipped with a partial composition law $\cdot \colon \bG^2 \dashrightarrow \bG$ and an inversion map ${}^{-1}\colon \bG \rightarrow \bG$, such that for all $f,g,h \in \bG$:
  \begin{enumerate}
  \item Composition is associative: $(fg)h = f(gh)$, as soon as one of the two sides is defined (which means that then the other is defined as well).
  \item The compositions $g^{-1} g$ and $g g^{-1}$ are always defined.
  \item If $fg$ is defined, then $f g g^{-1} = f$ and $f^{-1} f g = g$.
  \end{enumerate}

  We call $s_g = g^{-1} g$ the \emph{source} of $g$ and $t_g = g g^{-1}$ its \emph{target}.
  We call $e \in \bG$ \emph{neutral} if $e^2 = e$.
  The set of neutral elements of $\bG$ will be denoted $\bB$ or $\bB(\bG)$.
  We call $\bB$ the \emph{base} set of $\bG$, and say that $\bG$ is a groupoid \emph{over} $\bB$.
\end{dfn}

Let us make a few observations:

\begin{enumerate}
\item Both $s_g$ and $t_g$ are neutral for all $g \in \bG$, defining maps $s,t \colon \bG \rightarrow \bB$.
\item The composition $fg$ is defined if and only if $s_f = t_g$.
  In particular, $s_{fg} = t_{g^{-1}} = s_g$ and $t_{fg} = s_{f^{-1}} = t_f$.
\item If $e$ is neutral, then $e = e^{-1} e^2 = e^{-1} e = s_e$, and similarly $e = t_e$.
  In particular, $eg$ ($ge$) is defined, necessarily equal to $g$, if and only if $e = t_g$ ($e = s_g$).
\item If $fg$ is neutral, then $f = f g g^{-1} = g^{-1}$ and similarly $g = f^{-1}$.
  In particular, $(g^{-1})^{-1} = g$ and $(fg)^{-1} = g^{-1} f^{-1}$.
\end{enumerate}

From these observations follows the equivalence with the (possibly more familiar) definition of a groupoid as a category all of whose morphisms are invertible, in which case $\bB$ is the object set.

Notice that for $A \subseteq \bG$ we have $s(A) = A^{-1} A \cap \bB = A^{-1} \bG \cap \bB = \bG A \cap \bB$ and $t(A) = A A^{-1} \cap \bB = A \bG \cap \bB = \bG A^{-1} \cap \bB$.
Similarly, for $A \subseteq \bB$ we have $s^{-1}(A) = \bG A$ and $t^{-1}(A) = A \bG$.

The advantage of the algebraic definition is that it is easier to cast a topology on top of it.

\begin{dfn}
  \label{dfn:TopologicalGroupoid}
  A \emph{topological groupoid} is a groupoid such that $\bG$ is a Hausdorff topological space and composition and inversion are continuous (where defined).

  A topological groupoid $\bG$ with base $\bB$ is \emph{open} if the source map $s\colon \bG \rightarrow \bB$ is open.
\end{dfn}

We could also state the definition of a groupoid in a categorical language: an object equipped with arrows for source, inverse and product, say.
This would make the definition meaningful in any category with fibred products (and not only in the category of sets).
In the category of topological spaces and continuous maps, it would agree with our definition of a topological groupoid.

Since the source and target maps are total, the domain of composition, defined by the condition $t(g) = s(f)$, is closed in $\bG^2$.
It follows that the condition $g^2 = g$ is closed, so the base set $\bB$ is a closed subset of $\bG$.

Clearly, a topological groupoid is open if and only if its target map is open.
Every topological group, viewed as a groupoid over a point, is open.

\begin{dfn}
  \label{dfn:SpaceOver}
  A \emph{topological space over $\bB$} is a topological space $X$ equipped with a continuous map $\pi\colon X \rightarrow \bB$.
  The fibred product of two spaces over $\bB$ is
  \begin{gather*}
    X \times_\bB Y = \bigl\{ (x,y) \in X \times Y : \pi_X x = \pi_Y y \bigr\}.
  \end{gather*}
  When $X = \bG$ we take $\pi_X = s$, and when $Y = \bG$ we take $\pi_Y = t$.
\end{dfn}

In particular, the domain of composition in $\bG$ is $\bG \times_\bB \bG = \bigl\{ (g,h) \in \bG^2 : s_g = t_h \bigr\}$.

\begin{dfn}
  \label{dfn:GroupoidAction}
  Let $\bG$ be a groupoid over $\bB$, and $X$ a space over $\bB$.
  A \emph{continuous (left) action} of $\bG$ on $X$, denoted $\bG \curvearrowright X$, is a continuous map $\bG \times_\bB X \rightarrow X$, sending $(g,x) \mapsto gx$, such that $(gh)x = g(hx)$ whenever either is defined (so $\pi(gx) = t(g)$).
  A \emph{continuous right action} $X \curvearrowleft \bG$ is defined analogously as a map $X \times_\bB \bG \rightarrow X$.
\end{dfn}

In particular, the product map $\bG \times_\bB \bG \rightarrow \bG$ is both a left and a right continuous action of $\bG$ on itself.
On $\bB$, viewed as a space over itself, $\bG$ admits a unique action $(g,s_g) \mapsto t_g$ (and similarly a unique right action).

\begin{fct}
  \label{fct:OpenGroupoid}
  The following are equivalent for a topological groupoid $\bG$ over a base $\bB$:
  \begin{enumerate}
  \item The groupoid $\bG$ is open.
  \item For any topological space $X$ over $\bB$, the projection $\bG \times_\bB X \rightarrow X$ is open.
  \item For any continuous action $\bG \curvearrowright X$, the action law $\bG \times_\bB X \rightarrow X$ is open.
  \item The groupoid law $\bG \times_\bB \bG \rightarrow \bG$ is open.
  \end{enumerate}
\end{fct}
\begin{proof}
  \begin{cycprf}
  \item
    A basic open set of $\bG \times_\bB X$ is of the form $U \times_\bB V$, where $U \subseteq \bG$ and $V \subseteq X$ are open.
    Since $\bG$ is open, the sets $W = s(U) \subseteq \bB$ and $\pi^{-1}(W) \subseteq X$ are open, and the image of $U \times_\bB V$ in $X$ is the open set $V \cap \pi^{-1}(W)$.
  \item
    Compose with the homeomorphism $(g,x) \mapsto (g^{-1},gx)$.
  \item
    This is a special case.
  \item[\impfirst]
    If $U \subseteq \bG$ is open, then $U^{-1}U \subseteq \bG$ is open, and therefore $s(U) = U^{-1}U \cap \bB$ is open in $\bB$.
  \end{cycprf}
\end{proof}



\section{The groupoid associated to a classical theory}
\label{sec:GroupoidTheoryClassical}

In this Section, let $T$ denote a complete theory, in the sense of classical (i.e., not continuous) first order logic, in a countable language $\cL$.
We consider that by definition of the logic, all structures (so all models of $T$) are not empty.
In order to avoid borderline cases, let us also assume that no model of $T$ is a singleton (or, if $T$ is multi-sorted, that in no model are all sorts singletons).
By \emph{definable} we mean without parameters.

\begin{dfn}
  \label{dfn:ClassicalG0T}
  Let $T$ be a classical first-order theory in a countable language.
  Let $\bG_0(T) \subseteq \tS_{2 \times \bN}(T)$ consist of all possible types of a pair of enumerations of a model of $T$ (i.e., any two enumerations of any single countable model).
  Members of $\bG_0(T)$ will be denoted $g$, $h$, and so on, or possibly as types $g(x,y)$ where $x$ and $y$ stand for countable tuples of variables.
  Let $\bB_0(T) \subseteq \bG_0(T)$ to be the subset defined by the condition $x = y$.
  We may identify $\tp(a,a) \in \bB_0(T)$ with $\tp(a)$, thus identifying $\bB_0(T)$ with the subset of $\tS_\bN(T)$ consisting of types of enumerations of models.

  If $g = \tp(a,b)$ and $h = \tp(b',c')$, where $b \equiv b'$, then we might as well assume that $b = b'$, in which case $g^{-1} = \tp(b,a)$ and $gh = \tp(a,c')$ depend only on $g$ and $h$, and belongs to $\bG_0(T)$.
\end{dfn}

\begin{lem}
  \label{lem:ClassicalG0T}
  As defined above, $\bG_0(T)$ is a Polish open topological groupoid with base $\bB_0(T)$.
  If $g = \tp(a,b) \in \bG_0(T)$, then $t_g = \tp(a)$ and $s_g = \tp(b)$.
\end{lem}
\begin{proof}
  It is easy to check that $\bG_0(T)$ is indeed a topological groupoid.
  Let us prove that $\bG_0(T)$ is open, i.e., that the map $s \colon \tp(a,b) \mapsto \tp(b)$ is open.
  A basic open set $U \subseteq \bG_0(T)$ is defined by a formula $\varphi(x,y)$ (in which only finitely many variables actually appear).
  We claim that $s(U)$ is defined by $\exists x \, \varphi(x,y)$ (quantifying only over those $x_i$ that appear in $\varphi$).
  Indeed, let $\tp(b) \in \bB_0(T)$, so $b$ enumerate some $M \vDash T$.
  If $g = \tp(a,b) \in U$, then $a$ also enumerates $M$, $s(g) = \tp(b)$, and $\vDash \varphi(a,b)$ implies $\vDash \exists x \, \varphi(x,b)$.
  Conversely, if $\vDash \exists x \, \varphi(x,b)$, then there exists a tuple $a$ in $M$ such that $\vDash \varphi(a,b)$.
  Since only finitely many variables actually appear in $\varphi$, we may replace a tail of $a$ with an enumeration of $M$, so still $\vDash \varphi(a,b)$, and now $g = \tp(a,b) \in U$.
  Thus $s(U)$ is indeed defined by $\exists x \, \varphi(x,y)$.
\end{proof}

Our goal is to associate to each theory $T$ a groupoid $\bG(T)$ such that for any theory $T'$ we have $\bG(T) \cong \bG(T')$ as topological groupoids if and only if $T$ and $T'$ are bi-interpretable.
In fact, we desire a seemingly stronger version of the left-to-right implication, to which we refer as \emph{reconstruction}: a procedure by which we obtain, from $\bG(T)$, a theory bi-interpretable with $T$, in a (reasonably) constructive fashion.
While $\bG_0(T)$ may seem natural, neither implication seems to hold for it, nor, \textit{a fortiori}, reconstruction.
Indeed, naïve attempts at reconstruction quickly run into obstacles that seem to arise from the fact that the base $\bB_0(T)$ is not compact.
The following definition was originally an attempt to remedy this, i.e., to make the base compact.
Somewhat surprisingly, it solves all other issues at the same time, including that of a presenting the present work as a generalisation of the $\aleph_0$-categorical case.
An explanation of sorts as to why (rather than how) that happens is given in \autoref{sec:UniversalSkolemSort} (see \autoref{prp:UniversalSkolemSortClassical}).

We assume throughout that $T$ is in a single-sorted language.
The definitions and arguments adapt in an obvious manner to the multi-sorted case, with additional bookkeeping that we prefer to avoid.

\begin{dfn}
  \label{dfn:RichPhi}
  Assume that we work in a language in a single sort.
  A sequence $\Phi = \bigl( \varphi_n(x_{<n},y) : n \in \bN \bigr)$, where $y$ is a single variable, will be called \emph{rich} if every formula $\varphi(x_{<k},y)$ appears (with dummy variables) as $\varphi_n$ for some $n \geq k$.

  When there are many sorts, we fix a sort $S_i$ for each $x_i$, and require that in $\varphi_n(x_{<n},y)$, the variable $y$ belong to $S_n$.
\end{dfn}

Clearly, a rich $\Phi$ exists, provided, in the many-sorted case, that each sort is repeated infinitely often.

\begin{dfn}
  \label{dfn:DPhi}
  For a rich $\Phi = (\varphi_n : n \in \bN)$ we define
  \begin{gather*}
    D_{\Phi,n}(x_{<n}) = \bigwedge_{k<n} \forall y \, \bigl[ \varphi_k(x_{<k},y) \rightarrow \varphi_k(x_{\leq k}) \bigr],
    \qquad
    D_\Phi(x) = \bigwedge_{n \in \bN} D_{\Phi,n}(x_{<n}).
  \end{gather*}
  This just says that if there exists a witness for $\varphi_k$, then $x_k$ must be one.

  We shall view each $D_{\Phi,n}$ as a formula or as a definable set of $n$-tuples, as convenient.
  Similarly, $D_\Phi$ is a partial type or a type-definable set of infinite tuples.
\end{dfn}

\begin{lem}
  \label{lem:DPhiQuantifier}
  Any member of $D_{\Phi,n}$ in a countable model $M \vDash T$ can be extended to a member of $D_\Phi$ that moreover enumerates $M$.
  In particular, $D_\Phi$ is never empty.

  Moreover, let $\psi(x,y)$ be a formula, where $x$ is in the sort of $D_\Phi$ and $y$ is arbitrary (of course, only finitely many variables from the infinite tuples $x$ actually appear in $\psi$).
  Then the property
  \begin{gather*}
    (\exists x \in D_\Phi) \psi(x,y)
  \end{gather*}
  is expressible as a formula in the variables $y$.
\end{lem}
\begin{proof}
  The main assertion is immediate from the definition, and implies the moreover part.
\end{proof}

We can now associate to $T$ a groupoid through restriction of $\bG_0(T)$ to $D_\Phi$.

\begin{dfn}
  \label{dfn:ClassicalGT}
  Assume $T$ is a theory in classical logic, and let $\Phi$ be a rich sequence.

  We define $\tS_{m D_\Phi}(T) \subseteq \tS_{m \times \bN}(T)$ to be the (compact) set of possible types of members of the type-definable set $D_\Phi^m$.
  We define $\bB_\Phi(T) = \tS_{D_\Phi}(T)$, so $\bB_\Phi(T) \subseteq \bB_0(T)$ (any member of $D_\Phi$ must satisfy the Tarski-Vaught test), and $\bG_\Phi(T) = \bigl\{g \in \bG_0(T) : s_g,t_g \in \bB_\Phi(T) \bigr\}$.
\end{dfn}

In other words, $\bG_\Phi(T) = \bG_0(T) \cap \tS_{2 D_\Phi}(T)$ consists of all $\tp(a,b)$ where $a,b \in \Phi$ enumerate the same set (in fact, each is necessarily a sub-sequence of the other, repeating each element infinitely often).
The following is immediate from the definitions, but deserves nonetheless to be stated explicitly:

\begin{lem}
  \label{lem:ClassicalGT}
  As defined above, $\bG_\Phi(T)$ is Polish, open, and its base $\bB_\Phi(T)$ is the Cantor set.
\end{lem}
\begin{proof}
  The groupoid $\bG_\Phi(T)$ is Polish as a closed subset of a Polish space.
  The base $\bB_\Phi(T)$ is totally disconnected, compact and second-countable by construction.
  We have agreed to assume that no model of $T$ is a singleton, so no sentence that implies, modulo $T$, that the model is a singleton.
  This excludes the possibility of isolated points in $\bB_\Phi(T)$, which is therefore the Cantor set.
  Let $U \subseteq \bG_\Phi(T)$ be a basic open set, say defined by a formula $\chi(x,y)$.
  We claim that $s(U)$ is defined by $(\exists x \in D_\Phi) \chi(x,y)$ (following \autoref{lem:DPhiQuantifier}).
  Indeed, one inclusion is as for \autoref{lem:ClassicalG0T}, while the other follows from \autoref{lem:DPhiQuantifier}.
\end{proof}

Now things seem to be even worse: the groupoid depends not only on $T$, but also on $\Phi$.
Let us show that this is not truly a problem.

Let us fix some terminology.
The sorts of the language of $T$ will be called the \emph{basic sort(s)}.
More generally, an \emph{interpretable sort}, or, from now on, merely a \emph{sort}, will be any definable subset of a definable quotient of a product of the basic sorts:
\begin{gather*}
  S \subseteq (S_0 \times \ldots \times S_{n-1}) / E.
\end{gather*}
Say that a family of sorts is \emph{sufficient} if any sort (equivalently, any basic sort) is in a definable bijection with such a subset of quotient, with $S_i$ in the given family (so this family can be taken as an alternate family of basic sorts).
Of course, the easiest way to get a sufficient family of sorts is to take all basic sorts, together with some additional ones.

We gave \autoref{dfn:DPhi} with respect to the basic sorts, but we can just as well define with respect to any (sufficient) family of sorts.
So let us fix two rich sequences $\Phi$ and $\Psi$, with respect to two sufficient families of sorts (and we may reduce the general case to the one where one family is a superset of the other).

Let $x = (x_n)$ denote a variable in $D_\Phi$ and $y$ a variable in $D_\Psi$.
In what follows, $\exists x$ should be understood as $\exists x \in D_\Phi$, in the sense of \autoref{lem:DPhiQuantifier}, and similarly for $\forall x$, $\exists \tilde{x}$, as so on.
Similarly, we quantify on $y$ or $\tilde{y}$ over $D_\Psi$.

\begin{dfn}
  \label{dfn:ApproximateBijection}
  An \emph{approximate bijection} between $D_\Phi$ and $D_\Psi$ is a formula $\varphi(x,y)$ such that $\forall x \exists y \varphi$ and $\forall y \exists x \varphi$ are valid (i.e., consequences of $T$).
\end{dfn}

\begin{lem}
  \label{lem:DPhiQuantifierWitness}
  Let $\psi(x_{<n},y_{<m})$ be a formula, and assume that $(\exists y_{<m}) \psi$ is equivalent to $D_{\Phi,n}(x_{<n})$.
  Then there exist indices $n \leq i_0 < \ldots < i_{m-1}$ such that, letting $\bi = (i_j : j < m)$, the formula $\psi$ is equivalent to:
  \begin{gather*}
    (\exists z \in D_\Phi) \bigl( (z_{<n} = x_{<n}) \wedge (z_\bi = y_{<m}) \bigr).
  \end{gather*}
\end{lem}
\begin{proof}
  We choose $i_j$ by induction on $j < m$, such that $\varphi_{i_j}(x_{<i_j},z)$ is
  \begin{gather*}
    (\exists y_{<m})\left[ \psi(x_{<n},y_{<m}) \wedge (y_j = z) \wedge (y_{<j} = x_{i_{<j}}) \right].
  \end{gather*}
  With this choice, our assertion is easy to check.
\end{proof}

\begin{lem}
  \label{lem:ApproximateBijectionExtension}
  For any approximate bijection $\varphi$ between $D_\Phi$ and $D_\Psi$, and for any $j$, there exists a definable map $f\colon D_\Phi \rightarrow S_{y_j}$, where $S_{y_j}$ denotes the sort of $y_j$, such that $\varphi(x,y) \wedge \bigl( y_j = f(x) \bigr)$ is again an approximate bijection.
%
\end{lem}
\begin{proof}
  We may express the sort of $y_j$ as a definable subset of something of the form $(S_0 \times \cdots \times S_{m-1})/E$ for some basic sorts of $\Phi$ and definable equivalence relation $E$.
  Let $n$ be larger than any $i$ such that $x_i$ appears in $\varphi$.
  Let $\psi(x_{<n},\bar{z})$ be the formula
  \begin{gather*}
    D_{\Phi,n}(x_{<n}) \wedge \exists y \, \bigl( \varphi(x,y) \wedge y_j = [\bar{z}]_E \bigr).
  \end{gather*}
  Since $\varphi$ is assumed to be an approximate bijection, the formula $D_{\Phi,n}$ is equivalent to $(\exists \bar{z}) \psi$.
  In other words, $\psi$ satisfies the hypothesis of \autoref{lem:DPhiQuantifierWitness}, so let $\bi$ be as in the conclusion.
  We claim that $\varphi(x,y) \wedge \bigl( y_j = [x_\bi]_E \bigr)$ is an approximate bijection.
  Indeed, if $x \in D_\Phi$, then $\psi(x_{<n},x_{\bi})$ holds, so $y \in D_\Psi$ as desired exists.
  Conversely, if $y \in D_\Psi$, then a tuple $x_{<n}$ exists such that $D_{\Phi,n}(x_{<n}) \wedge \varphi(x_{<n},y)$ holds, and a tuple $\bar{z}$ exists such that $y_j = [\bar{z}]_E$.
  Therefore $\psi(x_{<n},\bar{z})$ holds, whence the existence of $x \in D_\Phi$ such that $x_\bi = \bar{z}$, and $y_j = [x_\bi]_E$.
  %
  %
  %
\end{proof}

\begin{prp}
  \label{prp:DPhiUnique}
  For any two sufficient families of sorts, and any two rich sequences $\Phi$ and $\Psi$ is these families, respectively, there exists a definable bijection $\sigma\colon D_\Phi \cong D_\Psi$.
\end{prp}
\begin{proof}
  For the main assertion, apply a back-and-forth construction using \autoref{lem:ApproximateBijectionExtension}.
  More precisely, start with $\varphi_0(x,y) = \top$ (True).
  Then, given $\varphi_n$, apply \autoref{lem:ApproximateBijectionExtension} twice to find $f_n$ and $g_n$ definable such that
  \begin{gather*}
    \varphi_{n+1}(x,y) = \varphi_n(x,y) \wedge x_n = f_n(y) \wedge y_n = g_n(x).
  \end{gather*}
  is an approximate bijection.
  Together, these yield the desired definable bijection.
  %
\end{proof}

One usually defines a \emph{bi-interpretation} between $T$ and $T'$ as a pair of interpretation schemes of one in the other, such that, when composed to yield an interpretation of $T$ or of $T'$ in itself, the models are uniformly definably isomorphic to their interpreted copies.
It is however fairly easy to check that this is equivalent to the property that the theory obtained by adjoining to $T$ the sort of $T'$ (without forgetting anything), and the one that is obtained by adjoining to $T'$ the sort of $T$, are the same up to a change of language.
This, together with \autoref{prp:DPhiUnique}, yields:

\begin{thm}
  \label{thm:GTWellDefined}
  Let $T$ and $T'$ be bi-interpretable, and let $\Phi$ and $\Psi$ be rich sequences for the languages of $T$ and of $T'$, respectively.
  Then $\bG_\Phi(T)$ and $\bG_\Psi(T')$ are isomorphic as topological groupoids.

  In other words, up to isomorphism of topological groupoids, $\bG_\Phi(T)$ does not depend on $\Phi$, and only depends on $T$ up to bi-interpretation.
\end{thm}

From now on we may denote $\bG_\Phi(T)$ by $\bG(T)$, omitting $\Phi$.

When $T$ is $\aleph_0$-categorical (so, in particular, complete), we have already associated to $T$ a different object, the topological \emph{group} $G(T) = \Aut(M)$, where $M$ is any countable model of $T$.
Viewing $G(T)$ as a topological groupoid, it is distinct from $\bG(T)$, since the base of $G(T)$ is a singleton (its identity).
Our next result says that this is the only difference between the two.

Let $G$ be a topological group and $\bB$ a topological space.
The set $\bB \times G \times \bB$ is naturally a groupoid based over $\bB$, with composition law $(x,g,y)(y,f,z) = (x,gf,z)$.

\begin{dfn}
  \label{dfn:TriviallyBasedGroupoid}
  Say that a topological groupoid $\bG$ is \emph{trivially based} if it is isomorphic, as a topological groupoid, to a groupoid of the form $\bB \times G \times \bB$ (where $\bB$ is necessarily the base of $\bG$).
  A \emph{trivialising section} for $\bG$ is a continuous map $\bg\colon \bB \rightarrow \bG$ such that $t \circ \bg = \id_\bB$ and $s \circ \bg$ is constant.
\end{dfn}

\begin{fct}
  \label{fct:TriviallyBasedGroupoid}
  A topological groupoid $\bG$ over $\bB$ is trivially based if and only if it admits a trivialising section.
  In this case $\bG \cong \bB \times G \times \bB$, where $G \cong \bG_e = \{g \in \bG : s_g = t_g = e\}$ for any $e \in \bB$.
\end{fct}
\begin{proof}
  Assume first that is trivially based, say $\bG = \bB \times G \times \bB$.
  Let $e \in \bB$.
  Then $G \cong \bG_e$, and $\bg(e') = (e',1,e)$ is a trivialising section.
  Conversely, assume that $\bg$ is a trivialising section, say $s \circ \bg \equiv e$, and let $G = \bG_e$.
  Then $f^\bg = \bg(t_f)^{-1} f \bg(s_f) \in G$ for all $f \in \bG$, and
  \begin{gather*}
    f \mapsto \bigl( t_f, f^\bg, s_f \bigr)
  \end{gather*}
  is the desired isomorphism $\bG \cong \bB \times G \times \bB$.
\end{proof}

\begin{prp}
  \label{prp:ClassicalGTAleph0Categorical}
  Let $T$ be $\aleph_0$-categorical, and let $G(T)$ be the isomorphism group of its countable model.
  Then $\bG(T) \cong 2^\bN \times G(T) \times 2^\bN$.
\end{prp}
\begin{proof}
  Let $\Phi$ be rich and let $q(x) \in \bB_\Phi(T)$.
  We have already observed that $\bB_\Phi(T) \cong 2^\bN$.
  Let $q_n(x_{<n})$ be the restriction of $q$ to $x_{<n}$, and for $\ell \geq n$, let $q_{n,\ell}(x_{<n},x_\ell)$ be the restriction of $q$ to $x_{<n},x_\ell$.
  We define $A_n \in \bN$ such that if $b \vDash q$, then any $1$-type over $b_{<n}$ is realised by $b_i$ for some $n \leq i < A_n$.
  We then define $B_0 = 0$ and $B_{k+1} = A_{B_k}$.
  Then, if $B_k \leq n < B_{k+1}$, we choose $m(n) > m(n-1)$ such that $\varphi_{m(n)}(x_{<m(n)},y)$ is the formula saying that
  \begin{itemize}
  \item if $q_{n+1}(x_{m(0),\ldots,m(n-1)},x_k)$ holds, then $y = x_k$,
  \item and otherwise, if $n < \ell < A_n$ is least such that $q_{n,\ell}(x_{m(0),\ldots,m(n-1)},x_k)$ holds (such $\ell$ must exist), then $q_{n+1,\ell}(x_{m(0),\ldots,m(n-1)},y,x_k)$.
  \end{itemize}
  Let $a \in D_\Phi$ and $b = m^*(a) = (a_{m(i)} : i \in \bN)$.
  One proves by induction on $n$ that $q_n( b _{<n} )$ must hold.
  Indeed, in the second case such a minimal $\ell$ must exist by choice of $A_n$, and in either case a $y$ as desired must exist, so $\varphi_{m(n)}(a_{<m(n)},b_n)$ holds, and implies $q_{n+1}(b_{\leq n})$.

  We also claim that $a$ and $b$ enumerate the same set, and more precisely, that $a_k = b_\ell$ for some $B_k \leq \ell < B_{k+1}$.
  Indeed, assume that $\tp(b_{<B_k},a_k) = q_{B_k,\ell}$, where $B_k \leq \ell < B_{k+1}$ is least.
  Then by induction on $B_k \leq n \leq \ell$ we have $\tp(b_{<n},a_k) = q_{n,\ell}$.
  In particular we have $\tp(b_{<\ell},a_k) = q_{\ell,\ell} = q_{\ell+1}$, so $b_\ell = a_k$.

  Therefore, if $p = \tp(a) \in \bB(T)$, then $\bg(p) = \tp\bigl( a, m^*(a) \bigr) \in \bG(T)$, and $\bg\colon \bB(T) \rightarrow \bG(T)$ is a trivialising section.
  It is easy to check that $\bG(T)_q \cong G(T)$, concluding the proof.
\end{proof}

\section{Reconstructing a classical theory}
\label{sec:ClassicalReconstruction}

We turn to reconstruction, namely, recovering $T$, up to bi-interpretation, from the topological groupoid $\bG = \bG(T) = \bG_\Phi(T)$, for some (any) choice of $\Phi$.
Members of $\bG$ represent $2$-types in $D_\Phi$, and we are soon going to see that we can recover formulas in two (imaginary sort) variables as subsets of $\bG$ -- most importantly, definable equivalence relations.
If we want to recover formulas in $k$ variables, we need an analogue of $\bG$ for $k$-types in $D_\Phi$.
This can be constructed directly from $\bG$ (that is to say, without knowing that it is of the form $\bG_\Phi(T)$), as follows.

\begin{dfn}
  \label{dfn:GroupoidFibredPower}
  Let $\bG$ be a topological groupoid, and let $k \in \bN$.
  We define $\bG^{k/t}$ as the $k$-fold $t$-fibred power (the first $e$ is not really necessary unless $k = 0$):
  \begin{gather*}
    \bG^{k/t} = \{ (e,g) \in \bB \times \bG^k : e = t_{g_0} = t_{g_1} = \cdots \}.
  \end{gather*}
  It is equipped with natural maps
  \begin{align*}
    t\colon \bG^{k/t} & \rightarrow \bB,
    & s\colon \bG^{k/t} & \rightarrow \bB^k,
    \\
    (e,g) & \mapsto e,
    & (e,g) & \mapsto (s_{g_0}, \ldots, s_{g_{k-1}}),
  \end{align*}
  and with corresponding groupoid actions actions $\bG \curvearrowright \bG^{k/t} \curvearrowleft \bG^k$.

  When $k \geq 1$, we define
  \begin{gather*}
    \bG^{[k]} = \bG \backslash \bG^{k/t} = \{ \bG g :g \in \bG^{k/t} \},
  \end{gather*}
  equipped with the quotient topology and the induced action $\bG^{[k]} \curvearrowleft \bG^k$.
\end{dfn}

We have a natural homeomorphism $\theta\colon \bG^{k/t} \cong \bG^{[k+1]}$:
\begin{gather*}
  \theta\colon (e,h) \mapsto \bG (e,e,h),
  \qquad
  \theta^{-1}\colon \bG (e, g) \mapsto (s_{g_0},g_0^{-1}g_1,\ldots, g_0^{-1}g_k).
\end{gather*}
This homeomorphism sends the actions $\bG \curvearrowright \bG^{k/t} \curvearrowleft \bG^k$ to $\bG^{[k+1]} \curvearrowleft \bG^{k+1}$:
\begin{gather*}
  \theta(g \cdot p \cdot h) = \theta(p) \cdot (g^{-1},h).
\end{gather*}
In particular, $\bB \cong \bG^{[1]}$ and $\bG \cong \bG^{[2]}$, replacing the double action $\bG \curvearrowright \bG \curvearrowleft \bG$ with $\bG \curvearrowleft \bG^2$ ($g \cdot (f,h) = f^{-1} g h$).

When $\bG = \bG_\Phi(T)$, it follows that $\bG^{[k]}$ can be identified with the space of types $p = \tp(a,b,c,\ldots) \in \tS_{k D_\Phi}(T)$ such that $a$, $b$, $c$, and so on all enumerate the same model.
Indeed, we may identify such $p$ with $(e,g) \in \bG^{k-1/t}$, where $e = \tp(a)$, $g_0 = \tp(a,b)$, $g_1 = \tp(a,c)$, and so on (it is easy to check that this identification is homeomorphic), and therefore with $\bG(e,e,g) \in \bG^{[k]}$.
From now on we shall just pretend that $\bG^{[k]}$ is given in this fashion as a subspace of $\tS_{k \Phi}(T)$, so $\bG = \bG^{[2]}$.
The action $\bG^{[k]} \curvearrowleft \bG^k$ is then easy to describe: if $p = \tp(a,b,\ldots) \in \bG^{[k]}$, $g = \tp(a,a')$, $h = \tp(b,b')$ and so on, then $p \cdot (g,h,\ldots) = \tp(a',b',\ldots)$.

Let also $\varphi(x^i : i < k)$ is a formula with $x^i \in D_\Phi$.
We then define
\begin{gather*}
  [\varphi] = \bigl\{ p \in \tS_{k D_\Phi}(T) : \varphi \in p \bigr\},
  \qquad
  [\varphi]_\bG = [\varphi] \cap \bG^{[k]} \subseteq \bG^{[k]}.
\end{gather*}
When $k = 2$, we may identify $\bG^{[2]}$ with $\bG$, and $\bigl[\varphi(x,y)\bigr]_\bG$ with a subset of $\bG$, accordingly.
Let us understand how the various actions above relate to this interpretation of formulas.

We say that $\varphi$ \emph{only uses $n$ variables} if of each $x^i$, which is an infinite tuple of variables, only the first $n$ ones, denoted $x^i_{<n}$, actually occur freely in $\varphi$.

\begin{lem}
  \label{lem:ClassicalReconstructionProductOfFormulas}
  Let $\varphi(x,y)$ and $\psi(y,z)$ be formulas with variables in $D_\Phi$, and let $\chi(x,z)$ be the formula $(\exists y) (\varphi \wedge \psi)$ (which is indeed a formula, as per \autoref{lem:DPhiQuantifier}).
  Then
  \begin{gather*}
    [\varphi]_\bG [\psi]_\bG = [\chi]_\bG,
  \end{gather*}
  where all are viewed as subsets of $\bG$.

  More generally, let $\varphi(x^i : i < k)$ be a formula, and for each $i < k$, let $\psi^i(x^i,y^i)$ be a formula, with all variables in $D_\Phi$.
  Let $\chi(y^i : i < k)$ be the formula
  \begin{gather*}
    (\exists x^0, x^1, \ldots )\bigl( \varphi \wedge \psi_0 \wedge \psi^1 \wedge \cdots \bigr).
  \end{gather*}
  Then
  \begin{gather*}
    [\varphi]_\bG \cdot \bigl( [\psi^0]_\bG \times [\psi^1]_\bG \times \cdots \bigr) = [\chi]_\bG,
  \end{gather*}
  Where $[\varphi]_\bG$ and $[\chi]_\bG$ are subsets of $\bG^{[k]}$, each $[\psi^i]_\bG$ is viewed as a subset of $\bG$, and the dot represents the action $\bG^{[k]} \curvearrowleft \bG^k$.
\end{lem}
\begin{proof}
  For the first identity, the inclusion $[\varphi]_\bG [\psi]_\bG \subseteq [\chi]_\bG$ is clear.
  For the opposite inclusion assume that $\tp(a,c) \in [\chi]_\bG$.
  Then $a$ and $c$ both enumerate the same model $M$.
  Assuming that $\varphi$ and $\psi$ only use $n$ variables, there exists a tuple $b_{<n} \in D_{\Phi,n}(M)$ such that $\varphi(a_{<n},b_{<n})$ and $\psi(b_{<n},c_{<n})$ hold.
  By \autoref{lem:DPhiQuantifier}, we may extend $b_{<n}$ to a sequence $b \in D_\Phi$ that enumerates $M$.
  Then $\tp(a,b) \in [\varphi]_\bG$, $\tp(b,c) \in [\psi]_\bG$, and their product is $\tp(a,c)$.

  The proof of the second, superficially more complex, case is essentially identical.
\end{proof}

Let $\tS_{k D_{\Phi,n}}(T)$ denote the space of $k$-types in $D_{\Phi,n}$.
We let $\pi = \pi_{k,n}\colon \tS_{k D_\Phi}(T) \rightarrow \tS_{k D_{\Phi,n}}(T)$ denote the natural projection $\tp(a,b,\ldots) \mapsto \tp(a_{<n},b_{<n},\ldots)$, and let $\pi_\bG = \pi_{k,n,\bG}\colon \bG^{[k]} \rightarrow \tS_{k D_{\Phi,n}}(T)$ denote its restriction to $\bG^{[k]}$.

\begin{lem}
  \label{lem:ClassicalReconstructionProjections}
  Let $k,n \in \bN$, $k \geq 1$.
  \begin{enumerate}
  \item The map $\pi\colon \tS_{k D_\Phi}(T) \rightarrow \tS_{k D_{\Phi,n}}(T)$ is continuous, closed, open and onto.
  \item We have $\pi(U) = \pi(U \cap \bG^{[k]})$ for every open $U \subseteq \tS_{k D_\Phi}(T)$.
  \item The restricted map $\pi_\bG\colon \bG^{[k]} \rightarrow \tS_{k D_{\Phi,n}}(T)$ is open and onto as well.
  \end{enumerate}
\end{lem}
\begin{proof}
  Continuity of $\pi$ (and therefore of $\pi_\bG$) is immediate, and together with compactness it implies that $\pi$ is closed.
  Openness of $\pi$ follows from the possibility to quantify (namely, \autoref{lem:DPhiQuantifier}): if $U = [\varphi] \subseteq \tS_{2 D_\Phi}(T)$ is a basic open set, then $\pi(U)$ is defined by the formula
  \begin{gather*}
    \psi(x_{<n},y_{<n}) = \bigl( \exists z,w \bigr) \bigl( \varphi(z,w) \wedge (x_{<n} = z_{<n}) \wedge (y_{<n} = w_{<n}) \bigr).
  \end{gather*}
  Onto follows from \autoref{lem:DPhiQuantifier}.

  Let $U \subseteq \tS_{k D_\Phi}(T)$ be open, and let $p = \tp(a^i : i < k) \in U$.
  Then there exists a formula $\varphi(x^i : i < k)$ such that $p \in [\varphi] \subseteq U$, and we may assume that that $\varphi$ only uses $m$ variables for some $m \geq n$.
  Let $M$ be a countable model containing all the $a^i$.
  By \autoref{lem:DPhiQuantifier}, there exist $b^i \in D_\Phi(M)$ that enumerate $M$, such that $b^i_{<m} = a^i_{<m}$.
  Then $q = \tp(b^i : i < k) \in [\varphi]_\bG \subseteq U \cap \bG^{[k]}$ and $\pi(p) = \pi(q) \in \pi_\bG(U \cap \bG^{[k]})$.

  It follows that $\pi_\bG$ is open and onto as well.
\end{proof}

Let $E^n(x,y)$ be the definable equivalence relation $x_{<n} = y_{<n}$ (where $x,y \in D_\Phi$).

\begin{lem}
  \label{lem:ClassicalReconstructionFormulaClopenEnInvariant}
  Let $n \geq 0$ and $k \geq 1$.
  Then the map
  \begin{gather*}
    \varphi(x^i : i < k) \mapsto [\varphi]_\bG
  \end{gather*}
  defines a bijection between formulas in $D_\Phi$ that only use $n$ variables (up to logical equivalence modulo $T$) and clopen subsets $X \subseteq \bG^{[k]}$ that are $[E^n]_\bG$-invariant, i.e., such that $X = X \cdot [E^n]_\bG^k$ (here $\cdot^k$ denotes Cartesian power).
\end{lem}
\begin{proof}
  Assume first that $\varphi(x^i : i < k)$ only uses $n$ variables, and let $X = [\varphi]_\bG$.
  Then it is clearly clopen in $\bG^{[k]}$, and it is $[E^n]_\bG$-invariant by \autoref{lem:ClassicalReconstructionProductOfFormulas}.
  It follows from \autoref{lem:DPhiQuantifier} that $\bG^{[[k]]}$ is dense in in $\tS_{k D_\Phi}(T)$.
  This implies in turn that if $[\varphi]_\bG = [\varphi']_\bG$, then $\varphi$ and $\varphi'$ must be equivalent modulo $T$.

  To see that the map is onto, let $X \subseteq \bG^{[k]}$ be clopen and $[E^n]_\bG$-invariant.
  Consider the map $\pi_\bG\colon \bG^{[k]} \rightarrow \tS_{k D_{\Phi,n}}(T)$, and let us prove that that $\pi_\bG(X) \cap \pi_\bG(\bG^{[k]} \setminus X) = \emptyset$.
  Indeed, assume that $p \in X$ and $q \in \bG^{[k]} \setminus X$ have the same image $\pi_\bG(p) = \pi_\bG(q)$.
  We may write $p = \tp(a^i : i < k)$ and $q = \tp(b^i : i < k)$, where the $a^i$ enumerate some model, and the $b^i$ enumerate another.
  The hypothesis $\pi_\bG(p) = \pi_\bG(q)$ means that $(a^i_{<n} : i < k) \equiv (b^i_{<n} : i < k)$, and we may assume that equality holds: $a^i_{<n} = b^i_{<n}$ for all $i < k$.
  Let $M$ be a countable model containing everything.

  Since $X$ is open, there exists a formula $\psi$ such that $p \in [\psi]_\bG \subseteq X$.
  Similarly, $X$ is closed, so there exists a formula $\chi$ such that $q \in [\chi] \subseteq X$, and we may assume that both $\psi$ and $\chi$ only use $m$ variables for some $m \geq n$.
  By \autoref{lem:DPhiQuantifier}, as usual, we may find $c^i$ and $d^i$ that enumerate $M$, such that $c^i_{<m} = a^i_{<m}$ and $d^i_{<m} = b^i_{<m}$.
  Let
  \begin{gather*}
    p' = \tp(c^i : i < k) \in [\psi]_\bG \subseteq X,
    \qquad
    q' = \tp(d^i : i < k) \in [\chi]_\bG \subseteq \bG^{[k]} \setminus X,
    \qquad
    g^i = \tp(c^i,d^i) \in [E^n]_\bG \subseteq \bG.
  \end{gather*}
  Then
  \begin{gather*}
    q' = p' \cdot (g^i : i < k) \in X \cdot [E^n]_\bG^k = X,
  \end{gather*}
  a contradiction.

  Thus, we have indeed proved that $\pi_\bG(X) \cap \pi_\bG(\bG^{[k]} \setminus X) = \emptyset$.
  Since $\pi_\bG$ is onto and open, it follows that $\pi_\bG(X)$ is clopen in $\tS_{k D_{\Phi,n}}(T)$.
  It is therefore defined by some formula $\varphi(x^i_{<n} : i < k)$.
  But then the same formula, with added dummy variables, defines $X$ in $\bG$, concluding the proof.
\end{proof}

The last technical step is to get rid of the hypothesis involving $E^n$ in \autoref{lem:ClassicalReconstructionFormulaClopenEnInvariant}.
Let $\sH$ denote the collection of clopen sub-groupoids of $\bG$ that contain $\bB$:
\begin{gather*}
  \sH = \bigl\{ \bH \subseteq \bG \ \text{clopen} : \bH = \bH \bH^{-1} \supseteq \bB \bigr\}.
\end{gather*}

\begin{lem}
  \label{lem:ClassicalReconstructionSubGroupoidContainsEn}
  Every $\bH \in \sH$ contains $[E^n]_\bG$ for some $n$.
\end{lem}
\begin{proof}
  Let $\bH \in \sH$.
  If $e \in \bB$, then $e \in \bH$, so $\bH$ contains a basic neighbourhood of $e$, i.e., one of the form $[\varphi]_\bG$.
  If $e = \tp(a)$, then $\varphi(a,a)$ must hold.
  If $\varphi$ only uses $n$ variables, then we may replace it with $\varphi(x,x) \wedge E^n(x,y)$.

  In other words, for each $e \in \bB$ there exist a formula $\varphi_e(x)$ and $n_e \in \bN$ such that
  \begin{gather*}
    e \in U_e = [ \varphi_e \wedge E^{n_e} ]_\bG \subseteq \bH.
  \end{gather*}
  By compactness, there is a finite family $e_i$ for $i < m$ such that $\bB \subseteq \bigcup_{i<m} U_{e_i}$.
  Then $\bB \subseteq \bigcup [\varphi_{e_i}]$, so
  \begin{gather*}
    [E^n]_\bG = \bigcup_{i<m} [\varphi_{e_i} \wedge E^n]_\bG \subseteq \bigcup_{i<m} U_{e_i} \subseteq \bH
  \end{gather*}
  where $n = \max n_{e_i}$.
\end{proof}

We can now reconstruct $T$ from $\bG$.
For this we need to recover
\begin{itemize}
\item the sorts of $T$, and
\item the formulas (definable subsets) on each finite product of sorts.
\end{itemize}

By \emph{sort} we mean any interpretable sort, as in the discussion following \autoref{lem:ClassicalGT}: indeed, we have no way to distinguish the basic sorts from the interpretable ones.
It follows from \autoref{prp:DPhiUnique} that any such sort is of the form (i.e., in definable bijection with) $D_{\Phi,n}/E$, for some $n$ and some definable equivalence relation $E$.
With some abuse of notation, we may even write it as $D_\Phi/E$, where $E(x,y)$ is again a definable relation in which only $x_{<n}$ and $y_{<n}$ actually appear.
The relation $E^n(x,y)$ which we defined earlier as $x_{<n} = y_{<n}$ is a definable equivalence relation, and any other definable equivalence relation on $D_\Phi$ coarsens of $E^n$ for some $n$.

\begin{lem}
  \label{lem:ClassicalReconstructionSorts}
  The map $E \mapsto [E]_\bG$ defines a bijection between definable equivalence relations on $D_\Phi$ and $\sH$.
  In addition, if $\bH = [E]_\bG$, $a \in D_\Phi$ enumerates $M$ and $e = \tp(a)$, then the map $\tp(a,b) \mapsto [b]_E$ (the $E$-class of $b$) is a bijection between the set
  \begin{gather*}
    e \bG / \bH = \{g \bH : t_g = e \}
  \end{gather*}
  and the sort $D_\Phi/E$ in $M$.
\end{lem}
\begin{proof}
  If $E$ is a equivalence relation on $D_\Phi$ and $\bH = [E]_\bG$, then it is easy to check that $\bH \in \sH$: in particular, $\bH \bH = \bH$ by \autoref{lem:ClassicalReconstructionProductOfFormulas}.
  Conversely, if $\bH \in \sH$, then by \autoref{lem:ClassicalReconstructionSubGroupoidContainsEn} and \autoref{lem:ClassicalReconstructionFormulaClopenEnInvariant} it is of the form $[E]_\bG$ for a unique formula $E(x,y)$.
  By the same reasoning, $E$ defines an equivalence relation: it is reflexive since $\bB \subseteq \bH$; it is symmetric since $\bH = \bH^{-1}$; and it is transitive since $\bH = \bH \bH$, using \autoref{lem:ClassicalReconstructionProductOfFormulas}.

  For the second part, if $a \in D_\Phi$ is a fixed enumeration of $M$, then any $g = \tp(a,b) \in \bG$ determines $b$, and $[b]_E \in D_\Phi/E$ in $M$.
  By \autoref{lem:DPhiQuantifier}, as usual, every member of $D_\Phi/E$ in $M$ is of this form.
  Finally, if $h = \tp(a,c) \in \bG$, then:
  \begin{gather*}
    [b]_E = [c]_E
    \quad \Longleftrightarrow \quad
    E(b,c)
    \quad \Longleftrightarrow \quad
    g^{-1} h = \tp(b,c) \in \bH
    \quad \Longleftrightarrow \quad
    g \bH = h \bH.
  \end{gather*}
  Therefore the map $g \bH \mapsto [b]_E$ is injective, completing the proof.
\end{proof}

Now that we have recovered the sorts, we may recover formulas.
Let $E_i$ be definable equivalence relations on $D_\Phi$ for $i < k$, and let $\bH_i = [E_i]_\bG$.

Say that a formula $\varphi(x^i : i < k)$ with $x^i \in D_\Phi$ is \emph{$E$-invariant} if it is $E_i$-invariant in each $x^i$.
Such a formula contains the exact same information as a formula $\tilde{\varphi}(\tilde{x}^i : i < k)$, with $\tilde{x}^i \in D_\Phi/E_i$, one being the pull-back of the other.
In this case, the set $[\varphi]_\bG \subseteq \bG^{[k]}$ is clopen and \emph{$\bH$-invariant}, that is to say that
\begin{gather*}
  [\varphi]_\bG = [\varphi]_\bG \cdot (\bH_0 \times \cdots \times \bH_{k-1}).
\end{gather*}

\begin{lem}
  \label{lem:ClassicalReconstructionFormulas}
  The map $\varphi \mapsto [\varphi]_\bG$ defines a bijection between $E$-invariant formulas (equivalently, formulas in the sorts $D_\Phi/E_0 \times \cdots \times D_\Phi/E_{k-1}$), up to logical equivalence modulo $T$, and $\bH$-invariant clopen subsets of $\bG^{[k]}$.

  Moreover, assume that $a \in D_\Phi$ enumerates a model $M$, let $e = \tp(a)$, and let us identify $D_\Phi/E_i$ in $M$ with $e \bG / \bH_i$ as per \autoref{lem:ClassicalReconstructionSorts}.
  In other words, a member $\tilde{b}_i$ of $D_\Phi/E_i$ is identified with $g_i \bH_i$, where $g_i = \tp(a,b_i)$ and $\tilde{b}_i = [b_i]_{E_i}$.
  Then $\tp(b_i : i < k) \in \bG^{[k]}$, and
  \begin{gather*}
    \varphi(\tilde{b}_i : i < k)
    \quad \Longleftrightarrow \quad
    \tp(b_i : i < k) \in [\varphi]_\bG.
  \end{gather*}
\end{lem}
\begin{proof}
  We have already observed that $[\varphi]_\bG$ is a clopen $\bH$-invariant set.
  For the converse direction, let $n$ be large enough that each $E_i$ only uses $n$ variables, and let $X \subseteq \bG^{[k]}$ be clopen and $\bH$-invariant.
  Then $X$ is also $[E^n]_\bG$-invariant, and therefore of the form $[\varphi]_\bG$ for a unique formula $\varphi(x^i : i < k)$, by \autoref{lem:ClassicalReconstructionFormulaClopenEnInvariant}.
  By \autoref{lem:ClassicalReconstructionProductOfFormulas}, $\varphi$ must be $E$-invariant.
  The moreover part is tautological.
\end{proof}

Together, \autoref{lem:ClassicalReconstructionSorts} and \autoref{lem:ClassicalReconstructionFormulas} tell us how to recover sorts, formulas, and their interpretations in countable models.

\begin{dfn}
  \label{dfn:ClassicalReconstruction}
  Let $\bG$ be an open groupoid over $\bB$, and let $\sH$ be the collection of clopen sub-groupoids of $\bG$ that contain $\bB$.
  Define a language $\cL(\bG)$ as follows:
  \begin{enumerate}
  \item It has one sort $D_\bH$ for each $\bH \in \sH$.
  \item It has a predicate symbol $P_X$ in the sorts $D_\bH = (D_{\bH_i} : i < k)$ for each sequence $\bH = (\bH_i : i < k)$ of such sub-groupoids and clopen, $\bH$-invariant $X \subseteq \bG^{[k]}$.
  \end{enumerate}

  For each $e \in \bB$ we define an $\cL(\bG)$-structure $M_e$.
  We interpret each sort $D_\bH$ as $e \bG/\bH = \{ g\bH : t_g = e\}$, and each predicate symbol $P_X$ as
  \begin{gather*}
    \bigl\{ (g_i \bH_i : i < k) : \bG (e , g) \in X \bigr\}.
  \end{gather*}
  Finally, we define $T(\bG)$ to be the $\cL(\bG)$-theory of the family $\{M_e : e \in \bB\}$.
\end{dfn}

The we have proven:

\begin{thm}
  \label{thm:ClassicalReconstruction}
  Let $T$ be a classical theory and $\bG = \bG(T)$.
  Then $T(\bG)$ is bi-interpretable with $T$.
  Up to a change of language, its sorts consist of all interpretable sorts in $T$, with the full induced structure.

  In particular, if $T'$ is another theory and $\bG(T) \cong \bG(T')$, then $T$ and $T'$ are bi-interpretable.
\end{thm}

\section{Universal Skolem sorts}
\label{sec:UniversalSkolemSort}

So far we have only treated the case of a theory in classical logic, even though the correspondence between $\aleph_0$-categorical theories and their automorphism groups, which we seek to generalise, also applies in continuous logic (see \cite{BenYaacov-Kaichouh:Reconstruction}).
Since we do not see how to generalise the construction of $D_\Phi$ to continuous logic, we shall follow here a different, more ``axiomatic'' path.

Throughout we work in the context of a complete theory $T$ in a countable language, in the sense of continuous logic.
By \emph{definable} (map, set, etc.) we always mean without parameters, unless explicitly said otherwise.
A \emph{sort} is any definable subset of an imaginary sort.
More precisely, the family of all metric sorts is generated by closing the basic sort(s) (i.e., those named in the language) under the following operations:
\begin{itemize}
\item \emph{Infinite product:} if $D_n$ is a sort for each $n$, then so is $\prod D_n$, equipped with any definable distance, say $d(x,y) = \sup_n 2^{-n} \wedge d(x_n,y_n)$.
  Formulas on an infinite product sort (i.e., with a variable in such a sort, and possibly other variables) are formulas on finite sub-products, as well as the uniform limits (so the proposed distance is indeed definable).
\item \emph{Metric quotient:} If $D$ is a sort and $d'$ is a definable pseudo-distance on $D$, then $D' = (D,d')$, obtained by dividing out the induced equivalence relation, is a sort as well.
  Notice that when applied to a structure that is not $\aleph_1$-saturated, one may also need to pass to the completion.
  Formulas on $D'$ are formulas on $D$ which are uniformly continuous with respect to $d'$.
  If $\varphi(x,y)$ is any formula with $x \in D$, then $\inf_{x'} \varphi(x',y) + N d'(x,x')$ ($N \in \bN$) is uniformly continuous (even Lipschitz) with respect to $d'$, and formulas obtained in this fashion are dense among all formulas in $D'$.
\item \emph{Subset:} If $D$ is a sort and $E \subseteq D$ is a definable subset, then $E$ is a sort as well.
  Formulas on $E$ are restrictions of formulas on $D$.
  Recall that $E \subseteq D$ is a definable set if the distance to $E$ is definable in $D$, or equivalently, if for every formula $\varphi(x,y)$, where $x \in D$, the expression $\psi(y) = \inf_{x \in E} \, \varphi(x,y)$ is again a formula.
\end{itemize}

By an easy compactness argument, any two definable distances on a sort are uniformly equivalent.
By the characterisation through quantifiers, the notion of a definable subset does not depend on the choice of a definable distance.
In addition, if $D$ is a sort $E \subseteq D$ is a definable subset, then any definable distance on $E$ extends to a definable pseudo-distance on $D$.
It follows that up to a definable isometric bijection, any sort, equipped with any definable distance, is a definable subset of a metric quotient of a product of the generating sorts.

Let us be given a theory $T$ in a language $\cL$, together with a family of sorts as defined above.
Let $\cL'$ extend $\cL$ with new basic sorts for the desired family of sorts, as well as new predicate symbols for formulas on any product of sorts (possibly restricting to a dense family of formulas).
Then there exists a unique theory $T'$ extending $T$ which says that the new basic sorts and new symbols interpret the desired sorts and formulas on them.
This adds no new additional structure on the original sorts (i.e., every formula is equivalent modulo $T'$ to an $\cL$-formula), and each of the new basic sorts admits a canonical definable bijection with the corresponding subset-of-quotient-of-product.

\begin{conv}
  Throughout, inequalities are interpreted with a universal quantifier in the context of a given theory $T$, so for example, $\inf_y \varphi(x,y) \leq r$ means that the sentence $\sup_x \inf_y \varphi(x,y) \leq r$ is a consequence of $T$.
\end{conv}

\begin{dfn}
  \label{dfn:SkolemMap}
  Let $D$ and $E$ be sorts, $\varphi(x,y)$ be a formula in $D \times E$, and $\varepsilon > 0$.
  An \emph{$\varepsilon$-Skolem map} for $\varphi$ is a definable map $\sigma\colon D \rightarrow E$ satisfying $\varphi(x,\sigma x) \leq \inf_y \varphi(x,y) + \varepsilon$.
\end{dfn}

One of the obstacles in continuous logic is that in general, one cannot name new Skolem maps in the language: there is no natural continuity modulus for such a map, and one can even construct examples where any such map would have to be discontinuous.

\begin{dfn}
  \label{dfn:UniversalSkolemSort}
  Let $D$ and $E$ be sorts.
  \begin{enumerate}
  \item We say that $D$ is a \emph{Skolem sort} for $E$ if every formula $\varphi(x,y)$ in $D \times E$ admits $\varepsilon$-Skolem maps for every $\varepsilon > 0$.
  \item We say that $D$ is \emph{universal} for $E$ if for every $\varepsilon > 0$ there exists a definable map $\sigma\colon D \rightarrow E$ such that the image of any $\varepsilon$-ball in $D$ is $\varepsilon$-dense in $E$.
  \end{enumerate}
  We say that $D$ is a \emph{Skolem} (\emph{universal}) sort if it is for every sort $E$.
\end{dfn}

If $\varphi(x,y)$ is any formula in $D \times E$, then it has the same Skolem maps as $\varphi(x,y) - \inf_z \varphi(x,z)$.
Therefore, we may restrict our attention to formulas satisfying $\inf_y \varphi = 0$.
It is also sufficient to test for existence of Skolem maps for a dense family of formulas in $D \times E$.
Combining the two observations, it suffices to test the existence of Skolem map on a dense subset of $\{\varphi : \inf_y \varphi = 0\}$.

Since any two definable distance on $E$ are uniformly equivalent, universality does not depend on any choice of definable distance.
A definable map has dense image in every model of $T$ if and only if it is surjective in any sufficiently saturated model.

\begin{lem}
  \label{lem:UniversalSkolemSortProduct}
  Let $D$ and $(E_m : m \in \bN)$ be sorts.
  Let $F_k = \prod_{m<k} E_m$ and $F = \prod_m E_m$.
  \begin{enumerate}
  \item If $D$ is Skolem for every $E_m$, then it is also for all $F_k$ and for $F$.
  \item If $D$ is universal for every $F_k$, then it is also for $F$.
  \end{enumerate}
\end{lem}
\begin{proof}
  First of all, either hypothesis implies that there exist definable maps $D \rightarrow E_m$ for every $m$.
  Therefore, any definable map $D \rightarrow F_k$ can be lifted into a definable map $D \rightarrow F$.

  Assume that $D$ is Skolem for $E$ and for $E'$ separately, and let $\varphi(x,y,y')$ be a formula in $D \times E \times E'$.
  Let $\sigma\colon D \rightarrow E$ be an $\varepsilon$-Skolem map for $\inf_{y'} \varphi(x,y,y')$, and let $\sigma'\colon D \rightarrow E'$ be an $\varepsilon$-Skolem map for $\varphi(x,\sigma x,y')$.
  Then $(\sigma,\sigma') \colon D \rightarrow E \times E'$ is a $2\varepsilon$-Skolem map for $\varphi$.
  It follows that if $D$ is Skolem for every $E_m$, then it is also Skolem for every $F_k$.
  Any formula in $D \times F$ can be approximated by a formula in $D \times F_k$, and an $\varepsilon$-Skolem map for the latter can be lifted to $F$ to give a, say, $2\varepsilon$-Skolem map for the former.

  For universality, we may equip $F_k$ and $F$ with the distance $d(y,y') = \sup_m 2^{-m} \wedge d(y_m,y'_m)$.
  If $2^{-k} < \varepsilon$, $\sigma\colon D \rightarrow F_k$ is definable, and any $\varepsilon$-ball in $D$ has $\varepsilon$-dense $\sigma$-image in $F_k$, then the same holds for any lifting of $\sigma$ to $D \rightarrow F$.
\end{proof}

\begin{lem}
  \label{lem:UniversalSkolemSortQuotient}
  Let $D$ and $E$ be sorts.
  If $D$ is a universal (Skolem) sort for $E$, then it is also for any quotient sort $F$ of $E$.
\end{lem}
\begin{proof}
  For universality, this follows from the quotient map $\pi\colon E \rightarrow F$ being uniformly continuous with dense image.
  For the Skolem property, just replace $\varphi(x,z)$ with $\varphi(x,\pi y)$.
\end{proof}

Let us now combine the two properties (universality and Skolem), to obtain a Skolem map which gets all potential witnesses (more or less).

\begin{dfn}
  \label{dfn:StrongSkolemMap}
  Let $\varphi(x,y)$ be a formula on $D \times E$ such that $\inf_y \varphi = 0$, and let $\sigma\colon D \rightarrow E$ be an $\varepsilon$-Skolem map for $\varphi$.
  We say that $\sigma$ is a \emph{combined} $\varepsilon$-Skolem map for $\varphi$ if for every $(a,b) \in D \times E$, if $\varphi(a,b) = 0$, then $d\bigl( \sigma B(a,\varepsilon), b \bigr) < \varepsilon$.
  It is \emph{strong} if under the same hypotheses, $b \in \sigma B(a,\varepsilon)$ in any model containing both $a$ and $b$ (and not merely in a saturated model).
\end{dfn}

\begin{lem}
  \label{lem:CombinedSkolemMap}
  Let $D$ and $E$ be sorts.
  Then $D$ is universal Skolem for $E$ if and only if, for every formula $\varphi(x,y)$ in $D \times E$ such that $\inf_y \varphi = 0$ and every $\varepsilon > 0$, there exists a combined $\varepsilon$-Skolem map $\sigma\colon D \rightarrow E$.
\end{lem}
\begin{proof}
  For right to left, the Skolem property is immediate, and for universality consider $\varphi = 0$.
  For the other direction, assume that $D$ is universal Skolem for $E$.
  Let $\delta > 0$ be small enough that $d(x,x'),d(y,y') < \delta$ imply $\bigl| \varphi(x,y) - \varphi(x',y') \bigr| < \varepsilon$, and by universality, let $\tau\colon D \rightarrow E$ be definable, such that the image of every $\delta$-ball is $\delta$-dense.
  Let $\psi(x,y)$ be the formula
  \begin{gather*}
    \varphi(x,y) + \bigl( 4 \varepsilon \dotminus \varphi(x,\tau x) \bigr) \wedge d(y, \tau x)
  \end{gather*}
  Considering the cases where $\varphi(x, \tau x) > 2 \varepsilon$ and $\leq 2 \varepsilon$ separately, we see that $\inf_y \psi \leq 2\varepsilon$ (in the first use the fact that $\inf_y \varphi = 0$, and in the second take $y = \tau x$).
  Let $\sigma$ be an $\varepsilon$-Skolem map for $\psi$.
  Then it is, in particular, a $3\varepsilon$-Skolem map for $\varphi$.

  Assume now that $\varphi(a,b) = 0$.
  By hypothesis on $\tau$, there exists $a' \in B(a,\delta)$ such that $d(\tau a',b) < \delta$.
  It follows that $\varphi(a', \tau a') < \varepsilon$.
  Since $\psi(a',\sigma a') \leq 3\varepsilon$, we must have $d(\sigma a', \tau a') \leq 3\varepsilon$.
  We conclude that $d\bigl( B(a,\delta), b \bigr) < 3\varepsilon + \delta$, which is enough.
\end{proof}

\begin{lem}
  \label{lem:StrongSkolemMap}
  Let $D$ and $E$ be sorts.
  Then the following are equivalent:
  \begin{enumerate}
  \item For every formula $\varphi(x,y)$ in $D \times E$ satisfying $\inf_y \varphi = 0$ and every $\varepsilon > 0$ there exists a strong $\varepsilon$-Skolem map $\sigma\colon D \rightarrow E$.
  \item There exists a sort $E' \supseteq E$ such that $D$ is universal Skolem for $E'$.
  \end{enumerate}
  In particular, being universal Skolem for $E$ passes to sub-sorts of $E$.
\end{lem}
\begin{proof}
  In one direction, Skolem is immediate and a strong $\varepsilon$-Skolem map for the zero formula yields (a strong variant of) universality.
  In the other direction, let $\varphi(x,y)$ and $\varepsilon > 0$ be given.
  Since $\varphi$ is always positive, we may extend $\varphi$ to a positive formula on $D \times E'$, denoted $\psi(x,y')$.
  In particular, $\inf_{y'} \psi = 0$ as well.
  Let $\eta_n = (1-2^{-n-1})\varepsilon$ and let $0 < \delta_n < \varepsilon/2^{n+2}$ be such that if $d(x_1,x_2) + d(y'_1,y'_2) \leq \delta_n$, then $\bigl| \psi(x_1,y'_1) - \psi(x_2,y'_2) \bigr| \leq \varepsilon/2^{n+3}$.

  We construct a sequence of definable maps $\sigma_n \colon D \rightarrow E'$, such that $d(\sigma_n x,E) \leq \delta_n$ and $\psi(x,\sigma_n x) \leq \eta_n$.
  We define
  \begin{align*}
    \psi_0(x,y') & = d(y',E) + \psi(x,y'),
    \\
    \psi_{n+1}(x,y') & = d(y',E) + \bigl[ \psi(x,y') \dotminus (\eta_n + \varepsilon/2^{n+3}) \bigr] + \bigl[ d(\sigma_n x,y') \dotminus \delta_n \bigr].
  \end{align*}
  We have $\inf_y \psi_0 = 0$ by assumption.
  Given $\sigma_n$, for each $a \in D$ there exists $b \in E$ such that $d(\sigma_n a,b) \leq \delta_n$, so $\psi(a,b) \leq \psi(a,\sigma_n a) + \varepsilon/2^{n+3} \leq \eta_n + \varepsilon/2^{n+3}$ and $\psi_{n+1}(a,b) = 0$.
  Therefore $\inf_y \psi_{n+1} = 0$ as well.

  By \autoref{lem:CombinedSkolemMap}, $\psi_n$ admits a combined $\delta_n$-Skolem map $\sigma_n\colon D \rightarrow E'$.
  Then indeed $d(\sigma_n x,E) \leq \delta_n$.
  We also have $\psi(x,\sigma_0 x) \leq \delta_0 < \eta_0$ and $\psi(x, \sigma_{n+1} x) \leq \eta_n + \varepsilon/2^{n+3} + \delta_{n+1} < \eta_{n+1}$, so the construction may proceed.

  We have $d(\sigma_n,\sigma_{n+1}) \leq \delta_n + \delta_{n+1}$, so the sequence $(\sigma_n)$ converges uniformly to a definable map $\sigma\colon D \rightarrow E'$.
  We have $d(\sigma x,E) \leq \lim \delta_n = 0$, so in fact $\sigma\colon D \rightarrow E$, and $\varphi(x,\sigma x) = \psi(x, \sigma x) \leq \lim \eta_n = \varepsilon$.

  Assume now that $\varphi(a,b) \leq 0$, and let us construct a sequence $(a_n) \subseteq D$ such that $\psi_n(a_n,b) = 0$.
  We start with $a_0 = a$ (indeed, $\psi_0(a,b) = 0$).
  Since $\sigma_n$ is combined $\delta_n$-Skolem for $\psi_n$ and $\psi(a_n,b) = 0$, there exists $a_{n+1} \in B(a_n,\delta_n)$ such that $d(b,\sigma_n a_{n+1}) < \delta_n$.
  We have $\varphi(a_n,b) \leq \eta_{n-1} + \varepsilon/2^{n+2} < \eta_n$, so $\varphi(a_{n+1},b) < \eta_n + \varepsilon/2^{n+3}$.
  Therefore $\psi_{n+1}(a_{n+1},b) = 0$, and the construction may proceed.

  The sequence $(a_n)$ converges to some $a' \in D$, where $d(a,a') < \sum \delta_n < \varepsilon$, and $\sigma a' = b$.
  If $a \in D(M)$ and $b \in E(M)$ for some $M \vDash T$, then the entire sequence can be constructed in $D(M)$, proving that $\sigma$ is a strong $\varepsilon$-Skolem function for $\varphi$.
\end{proof}

\begin{prp}
  \label{prp:UniversalSkolemSort}
  A sort $D$ is universal Skolem (for all sorts) if and only if it is Skolem for every basic sort, and universal for any finite product of the basic sort(s).
\end{prp}
\begin{proof}
  One direction is immediate, and the other follows from \autoref{lem:UniversalSkolemSortProduct}, \autoref{lem:UniversalSkolemSortQuotient} and \autoref{lem:StrongSkolemMap}.
\end{proof}

Let $D$ and $E$ be any two sorts.
We equip the space of definable maps $\sigma\colon D \rightarrow E$ with the distance of uniform convergence
\begin{gather*}
  d(\sigma,\rho) = \sup_{x \in D} \, d(\sigma x, \rho x).
\end{gather*}
This renders the space of definable maps a complete separable metric space.

\begin{thm}
  \label{thm:UniversalSkolemBijection}
  Let $D$ and $E$ be two sorts that are universal Skolem for each other.
  Then there exists a definable bijection $\sigma\colon D \cong E$.
\end{thm}
\begin{proof}
  We construct surjective definable maps $\sigma_n \colon E \rightarrow D$ and $\rho_n\colon D \rightarrow E$ as follows.
  We start with $\sigma_0$, which exists by universality of $E$.

  Assume now that $\sigma_n$ is known.
  Let $\varphi_n(x,y)$ be the formula $d(x,\sigma_n y)$.
  Then $\inf_x \varphi_n = 0$, and since $\sigma_n$ is surjective, $\inf_y \varphi_n = 0$ as well.
  If $n = 0$, let $0 < \varepsilon_0 < 1$ be arbitrary.
  For $n > 0$, since a definable map is uniformly continuous, choose $0 < \varepsilon_n < 2^{-n}$ such that
  \begin{gather*}
    d(x,x') < \varepsilon_n
    \quad \Longrightarrow \quad
    d(\rho_{n-1} x, \rho_{n-1} x'') < 2^{-n}.
  \end{gather*}
  Then choose a strong $\varepsilon_n$-Skolem function $\rho_n\colon D \rightarrow E$ for $\varphi_n$.
  Since $\rho_n$ Skolem, we have
  \begin{gather*}
    d(x, \sigma_n \rho_n x) = \varphi_n(x, \rho_n x) < \varepsilon_n.
  \end{gather*}
  Since  $\varphi_n(\sigma_n y,y) = 0$ and $\rho_n$ is strong, it is surjective.

  Similarly, given $\rho_n\colon D \rightarrow E$ we construct a surjective definable $\sigma_{n+1}\colon E \rightarrow D$ such that
  \begin{gather*}
    d(y, \rho_n \sigma_{n+1} y) < \delta_n < 2^{-n},
  \end{gather*}
  where
  \begin{gather*}
    d(y,y') < \delta_n
    \quad \Longrightarrow \quad
    d(\sigma_n y, \sigma_n y') < 2^{-n}.
  \end{gather*}

  Once the construction is complete, we have
  \begin{gather*}
    d(\rho_n,\rho_{n+1})
    \leq d(\rho_n,\rho_n \sigma_{n+1} \rho_{n+1}) + d(\rho_n \sigma_{n+1} \rho_{n+1}, \rho_{n+1})
    < 2^{-n-1} + 2^{-n},
    \\
    d(\sigma_n,\sigma_{n+1})
    \leq d(\sigma_n,\sigma_n \rho_n \sigma_{n+1}) + d(\sigma_n \rho_n \sigma_{n+1}, \sigma_{n+1})
    < 2^{-n} + 2^{-n}.
  \end{gather*}
  The sequences $(\sigma_n)$  and $(\rho_n)$ converge uniformly to definable maps $\sigma$ and $\rho$, and $\rho = \sigma^{-1}$.
\end{proof}

In particular, the universal Skolem sort, if it exists, is unique (up to a definable bijection).
Let us point out a few general properties of universal Skolem sorts.

\begin{lem}
  \label{lem:UniversalSkolemSortTypes}
  Let $D$ be a universal Skolem sort.
  The space of types in $D$, denoted $\tS_D(T)$, is homeomorphic to the Cantor space.
  Moreover, if $U \subseteq \tS_D(T)$ is clopen and non-empty, and $D_U \subseteq D$ consists of all realisations of types in $U$, then $D_U$ is definable in $D$, and is again a universal Skolem sort.
\end{lem}
\begin{proof}
  Assume that $p,q \in \tS_D(T)$ are distinct.
  Then there exists a formula $\varphi(x)$, say with values in $[0,1]$, such that $\varphi(p) = 0$ and $\varphi(q) = 1$.
  Let $y$ be a variable in the sort $\{0,1\}$, and define $\psi(x,y)$ to be $2\varphi(x) \dotminus 1$ if $y = 0$ and $1 \dotminus 2 \varphi(x)$ if $y = 1$, so $\inf_y \psi = 0$.
  If $\sigma\colon D \rightarrow \{0,1\}$ is $1/3$-Skolem, then it separates the type space into two clopen sets, one containing $p$ and the other $q$.
  This proves that $\tS_D(T)$ is totally disconnected.

  Let $U,V \subseteq \tS_D(T)$ be non-empty, complementary clopen sets.
  By a compactness argument, $d(D_U,D_V) = r > 0$, so $D_U$ is a definable subset of $D$.
  Considering $r > \varepsilon > 0$, we see that $D_U$ is also universal, and it is clearly Skolem.

  Since a universal sort must realise more than one type, this also shows that $\tS_D(T)$ has no isolated points.
  Being metrisable (since the language is countable), it is the Cantor set.
  %
\end{proof}

\begin{lem}
  \label{lem:UniversalSkolemSortInverseLimit}
  If $D_0 \twoheadleftarrow D_1 \twoheadleftarrow \cdots$ is an inverse system of universal Skolem sorts with surjective definable maps, then its inverse limit is again universal Skolem sort.
\end{lem}
\begin{proof}
  First of all, it is fairly easy to check that the inverse limit, call it $D$, is a definable subset of $\prod D_i$, so it is a sort.
  The maps $D \rightarrow D_i$ are definable and surjective, and since each $D_i$ is universal, any one of them can be used to show that $D$ is universal as well.
  For any sort $E$, any formula on $D \times E$ can be approximated arbitrarily well by a formula on $D_i \times E$ for some $i$ large enough, so $D$ is also Skolem.
\end{proof}

\begin{lem}
  \label{lem:UniversalSkolemSortAbsorption}
  Let $D$ be a universal Skolem sort of $T$.
  Then $D \times 2$ and $D \times 2^\bN$ are also universal Skolem sorts.
\end{lem}
\begin{proof}
  It follows from \autoref{lem:UniversalSkolemSortTypes} and the uniqueness of the universal Skolem sort that $D$ admits a definable bijection with $D \times 2$.
  Now apply \autoref{lem:UniversalSkolemSortInverseLimit} to the inverse system consisting of $D \times 2^n$.
\end{proof}

\begin{lem}
  \label{lem:UniversalSkolemSortModel}
  Let $D$ be the universal Skolem sort of $T$.
  Then every $a \in D$ is interdefinable with a model, necessarily separable.
  Conversely, if $M \vDash T$ is a separable model, then the set of $a \in D(M)$ that are interdefinable with $M$ is dense in $D(M)$.
\end{lem}
\begin{proof}
  Assume that $M \vDash T$ and $a \in D(M)$.
  Let $N \subseteq M$ be the definable closure of $a$ in the basic sort(s).
  The existence of Skolem maps implies that $N \preceq M$ (by the Tarski-Vaught Criterion) and that $a$ and $N$ are interdefinable.

  Now let us assume that $M$ is separable, and let $b$ be an enumeration of a dense countable sequence in $M$.
  Let $E$ denote the sort of $b$.
  By universality, for every $\varepsilon > 0$ there exists a definable map $\sigma\colon D \rightarrow E$ such that for all $a \in D(M)$ there exists $a' \in B(a,\varepsilon) \cap D(M)$ such that $\sigma a' = b$.
  Such $a'$ is necessarily interdefinable with $M$, proving density.
\end{proof}

Let us pass to the question of the existence of a universal Skolem sort.
First of all, one need not always exist, as the following (admittedly pathological) example shows.

\begin{exm}
  \label{exm:NoSkolemSort}
  Let $\cL$ be a continuous signature, with bound one on the diameter, and a single unary $1$-Lipschitz $[0,1]$-valued predicate symbol $P$.
  Let $T$ be the theory saying that the distance is always either $0$ or $1$ and $P$ has dense image (i.e., the sentences $\sup_{x,y} d(x,y) \bigl( 1 - d(x,y) \bigr)$ and $\inf_x |P(x) - r|$ vanish for every $r \in [0,1]$).
  In any sufficiently saturated model of $T$, each $r \in [0,1]$ is attained as $P(x)$ for infinitely many possible values of $x$, and a back-and-forth argument between two such models shows that $T$ eliminates quantifiers.
  In particular, $T$ is complete, and $\tS_1(T)$ is the interval $[0,1]$.

  Assume that $T$ admits a Skolem sort $D$, and let $E$ denote the home sort.
  Then there exists a map $\sigma\colon D \rightarrow E$ such that $P(\sigma x) < 1/2$.
  Since $D$ is a sort and $\sigma$ is definable, $\varphi(x) = d(x,\img \sigma)$ is a formula (we may also express it as $\inf_{y \in D} \, d(x,\sigma y)$).
  It is $0/1$-valued, so it cuts $\tS_1(T) = [0,1]$ into two non-trivial clopen sets, a contradiction.

  Therefore $T$ cannot admit a Skolem sort.
\end{exm}

Our definition of a universal Skolem sort was motivated by $D_\Phi$ of \autoref{sec:GroupoidTheoryClassical}.
Let us now justify this formally.

\begin{prp}
  \label{prp:UniversalSkolemSortClassical}
  Assume that $T$ is classical.
  Then viewed as a theory in continuous logic, the set $D_\Phi$, constructed in \autoref{dfn:ClassicalGT}, is a universal Skolem sort.
\end{prp}
\begin{proof}
  Assume that $T$ is single-sorted, for simplicity of the definition of $D_\Phi$ and of the argument presented here.
  That $D_\Phi$ is a definable set, i.e., a sort, follows immediately from the fact that for each $n$, the set of $n$-tuples which can be extended to a member of $D_\Phi$, is definable (by $D_{\Phi,n}$).
  The sort $D_\Phi$ is Skolem for the basic sort $E$ by construction.

  For universality, we may assume that $D_\Phi$ is equipped with the distance $d(x,y) = \inf \, \bigl\{ 2^{-n} : x_{<n} = y_{<n} \bigr\}$.
  Given any $k$ and $\varepsilon > 0$, we may choose $m_0 < m_1 < \cdots < m_{k-1}$ such that each $\varphi_{m_i}$ is always true and $2^{-m_0} < \varepsilon$.
  Then the map $D_\Phi \rightarrow E^k$ that sends $x \mapsto (x_{m_i} : i < k)$ is surjective on any $\varepsilon$-ball, so $D_\Phi$ is universal for $E^k$.
  By \autoref{prp:UniversalSkolemSort}, this is enough.
\end{proof}

When $T$ is $\aleph_0$-categorical, classical or continuous, we can give another construction of a universal Skolem sort.
It generalises \autoref{prp:ClassicalGTAleph0Categorical} to the continuous case (and, in a sense, explains it better).

\begin{prp}
  \label{prp:UniversalSkolemSortAleph0Categorical}
  Assume that $T$ is $\aleph_0$-categorical.
  Let $a$ enumerate a dense subset of a model $M \vDash T$, let $D_0$ be the type of $a$ (a definable set, since $T$ is $\aleph_0$-categorical).
  Then $D = D_0 \times 2^\bN$ is a universal Skolem sort.
\end{prp}
\begin{proof}
  It will suffice to show that $D$ is universal Skolem for every sort $E$.
  Let a variable in $D$ be denoted $\hat{x} = (x,\tilde{x})$, where $x \in D_0$ and $\tilde{x} \in 2^\bN$.

  In order to show that $D$ is a Skolem sort, let $\varphi(\hat{x},y)$ be a formula on $D \times E$ such that $\inf_y \varphi = 0$.
  We may assume that $\varphi$ only depends on the first $k$ entries of $\tilde{x}$ (by density of such formulas).
  In other words, we may view $\varphi$ as a formula on $D_0 \times 2^k \times E$, and write $\varphi(x,\ell,y)$ where $\ell < 2^k$.
  For each $\ell < 2^k$, choose $b_\ell \in E(M)$ such that $\varphi(a,\ell,b_\ell) < \varepsilon$.
  Let $\sigma\colon D_0 \times 2^k \rightarrow E$ be the map which sends $(a,\ell) \mapsto b_\ell$ (and $(a',\ell)$ to the unique $b'$ such that $a'b' \equiv a b_\ell$).
  Then $\sigma$ is definable, and we may view it as a map $\sigma\colon D \rightarrow E$ that only depends on the first $k$ bits.
  It is $\varepsilon$-Skolem by construction.

  In order to show that $D$ is universal, let us fix $\varepsilon$.
  By the Ryll-Nardzewski/Henson characterisation of $\aleph_0$-categoricity (see \cite{BenYaacov-Usvyatsov:dFiniteness}), the type space $\tS_{D_0 \times E}(T)$ is metrically compact, so it contains a finite, $\varepsilon$-dense sequence $(p_\ell : \ell < 2^k)$.
  Let $p_\ell = \tp(a_\ell,b_\ell)$.
  We may choose $a_\ell' \in D_0$ such that $b_\lambda \in \dcl(a_\ell')$ and such that $d(a_\lambda,a_\ell')$ is arbitrarily small.
  We may therefore assume that $b_\ell \in \dcl(a_\ell)$, and in fact that $a_\ell = a$ and $b_\ell \in E(M)$ for all $\ell$.
  Define $\sigma\colon D \rightarrow E$ as in the previous paragraph.
  Now, for any $b \in E(M)$, there exist $\ell$ and $a',b' \vDash p_\ell$, possibly outside $M$, such that $d(a'b',a b_\ell) < \varepsilon$.
  In particular, $\sigma(a',\ell) = b'$, so $\inf_x  d(x,a) \vee d\bigl(\sigma(x,\ell), b\bigr) < \varepsilon$ (we use $\vee$ as infix notation for the maximum).
  This is almost good enough: if we code $\ell$ not in the first $k$ bits, but sufficiently farther along the infinite sequence that is $\tilde{x}$, we obtain, for any $\tilde{a} \in 2^\bN$:
  \begin{gather*}
    \inf_{x,\tilde{x}}  d(x,a) \vee d(\tilde{x},\tilde{a}) \vee d\bigl(\sigma(x,\tilde{x}), b \bigr) < \varepsilon,
  \end{gather*}
  concluding the proof.
\end{proof}

When a universal Skolem sort exists, it allows us to associate to $T$ a canonical (or almost) bi-interpretable theory.

\begin{dfn}
  \label{dfn:TD}
  Let $T$ be a theory and $D$ a universal Skolem sort.
  We define $T^D$ to be the theory of the sort $D$ together with the induced structure.
\end{dfn}

The full induced structure on $D$ is given by naming all formulas with variables in $D$ by predicate symbols.
Since the language of $T$ is assumed countable, the set of all $n$-ary formulas is separable for each $n$, and naming a countable dense subset is just as good.

\begin{lem}
  \label{lem:TD}
  Let $T$ be a theory admitting a universal Skolem sort.
  Then $T^D$ is bi-interpretable with $T$.
  Conversely, up to choice of language, and in particular of distance (among all definable distances), the theory $T^D$ only depends on the bi-interpretation class of $T$, and in particular, does not depend on the choice of universal Skolem sort.
\end{lem}
\begin{proof}
  Consider the theory $T'$ consisting of $T$ with all its basic sorts, together with $D$ as an additional sort, and all the induced structure on the entire family of sorts.
  This is an interpretation expansion of both $T$ (since $D$ is a sort) and of $T^D$ (since all sorts are quotients of $D$), so $T$ and $T^D$ are bi-interpretable.
  Independence on the choice of $D$ follows from \autoref{thm:UniversalSkolemBijection}.
\end{proof}

\section{The groupoid associated to a theory with a universal Skolem sort}
\label{sec:GroupoidTheory}

Before introducing any hypotheses, let us prove the following technical fact.

\begin{lem}
  \label{lem:DefinableElement}
  Let $T$ be any theory in a countable language.
  \begin{enumerate}
  \item
    \label{item:DefinableElementPolish}
    Let $A$ and $B$ be any two sorts of $T$, and let $X \subseteq \tS_{A,B}(T)$ be the set of types $\tp(a,b)$, where $a \in A$, $b \in B$, and $b$ is definable from $a$.
    Then $X$ is a $G_\delta$ subset of $\tS_{A,B}(T)$.
  \item
    \label{item:DefinableElementComposition}
    Let $C$ be an additional sort, and let $Y \subseteq \tS_{B,C}(T)$ be the set of types $\tp(b,c)$, where $b \in B$, $c \in C$, and $c$ is definable from $b$.
    Let $X \times_B Y$ consist of all pairs $(p,q)$ that agree on the type of the member of $B$.
    Any such pair can be written as $\bigl( \tp(a,b), \tp(b,c) \bigr)$, in which case $c$ is definable from $a$, and we may define a composition $p \circ q = \tp(a,c)$.
    Then $\circ \colon X \times_B Y \rightarrow \tS_{A,C}(T)$ is continuous.
  \end{enumerate}
\end{lem}
\begin{proof}
  For $\varepsilon > 0$, and formula $\varphi(x,y)$ in $A \times B$, let $U_{\varepsilon,\varphi} \subseteq \tS_{A,B}(T)$ be the open set defined by
  \begin{gather*}
    \varphi(x,y) \vee \sup_{z,z'} \, \bigl( d(z,z') - \varphi(x,z) - \varphi(x,z') \bigr) < \varepsilon.
  \end{gather*}
  Let
  \begin{gather*}
    V_\varepsilon = \bigcup_\varphi U_{\varepsilon,\varphi},
    \qquad
    W = \bigcap_{\varepsilon > 0} V_\varepsilon.
  \end{gather*}

  Let $p(x,y) = \tp(a,b) \in X$.
  Then $d(y,b)$ is definable with parameter $a$, i.e., $d(y,b) = \varphi(a,y)$ for some formula $\varphi(x,y)$, in which case $p \in U_{\varepsilon,\varphi}$ for all $\varepsilon > 0$, and therefore $p \in W$.
  Conversely, assume that $p \in W$, and let $\varepsilon > 0$.
  Then there exists a formula $\varphi$ such that $p \in U_{\varepsilon,\varphi}$.
  But then the diameter of the set of realisations of $p(a,y)$ is at most $3\varepsilon$, and since $\varepsilon$ was arbitrary, $b$ is the unique realisation of $p(a,y)$, so $p \in X$.
  We conclude that $W = X$, and it is $G_\delta$ by construction.

  Let again $p(x,y) = \tp(a,b) \in X$, and let $q(y,z) = \tp(b,c) \in Y$, so $p \circ q = \tp(a,c)$.
  A neighbourhood of $\tp(a,c)$ can be assumed to be defined by a condition $\varphi(x,z) < 1$, where $\varphi(a,c) = 0$.
  Since $c$ is definable from $b$, we may express $\varphi(x,c)$ as $\psi(x,b)$.
  Let
  \begin{gather*}
    \chi(y,z) = \sup_x \, \bigl| \varphi(x,z) - \psi(x,y) \bigr|.
  \end{gather*}
  Then $\psi(a,b) = \chi(b,c) = 0$, and
  \begin{gather*}
    \bigl(X \cap [\psi < 1/2] \bigr) \circ \bigl(Y \cap [\chi < 1/2] \bigr) \subseteq [\varphi < 1].
    \qedhere
  \end{gather*}
\end{proof}

From this point onward, assume that $T$ is a complete theory in a countable continuous language, admitting a universal Skolem sort $D$.
We let $\tS_{mD}(T)$ denote the space of types in $m$ variables in the sort $D$ (i.e., in $D^m$).

\begin{dfn}
  \label{dfn:GT}
  We define $\bG(T)$ (or $\bG_D(T)$, if we want to be explicit) as the set of all types $\tp(a,b) \in \tS_{2 D}(T)$ such that $\dcl(a) = \dcl(b)$.
  We shall implicitly identify a type $\tp(a,a) \in \bG(T)$ with $\tp(a)$, and let $\bB(T) = \tS_D(T)$ be the collection of all such types.

  The groupoid structure is defined as in \autoref{dfn:ClassicalG0T}:
  \begin{gather*}
    \tp(a,b) \cdot \tp(b,c) = \tp(a,c),
    \qquad
    \tp(a,b)^{-1} = \tp(b,a).
  \end{gather*}
\end{dfn}

\begin{prp}
  \label{prp:GT}
  As defined in \autoref{dfn:GT}, $\bG(T)$ is an open Polish topological groupoid.
  Its base is $\bB(T)$, which is homeomorphic to the Cantor set, and the action $\bG(T) \curvearrowright \bB(T)$ is minimal (i.e., all orbits are dense).
  As a topological groupoid, $\bG(T)$ only depends on the bi-interpretation class of $T$ (in particular, it does not depend on $D$).
\end{prp}
\begin{proof}
  It is easy to check that $\bG(T)$ is a groupoid over $\bB(T)$, with source and target maps given by
  \begin{gather*}
    g = \tp(a,b)
    \qquad
    \Longrightarrow
    \qquad
    t_g = \tp(a), \quad s_g = \tp(b).
  \end{gather*}
  It is a Polish topological groupoid by \autoref{lem:DefinableElement}, and $\bB(T)$ is homeomorphic to the Cantor space by \autoref{lem:UniversalSkolemSortTypes}.
  Since the universal Skolem sort is unique up to a definable bijection, $\bG(T)$ only depends on the bi-interpretation class of $T$.

  To see that $\bG(T) \curvearrowright \bB(T)$ is minimal, let $V = [\varphi(x) > 0] \subseteq \bB(T)$ be a non-empty basic open set, and let $e \in \bB(T)$.
  Then $e = \tp(a)$ for some $a \in D$, which codes a separable model $M$.
  Since $T$ is complete and $V \neq \emptyset$, $T$ must imply that $\sup_x \varphi > 0$, and so there exists $b \in D(M)$ such that $\varphi(b) > 0$.
  By the density clause in \autoref{lem:UniversalSkolemSortModel}, there exists $c \in D(M)$ arbitrarily close to $b$ that codes $M$ as well.
  Taking $d(b,c)$ small enough we have $\varphi(c) > 0$, and $g = \tp(c,a) \in \bG(T)$ sends $e$ into $V$.

  To see that $\bG(T)$ is open, let $U = \bigl[ \varphi(x,y) > 0 \bigr] \subseteq \bG(T)$ be a basic open set, and let $V \subseteq \bB(T)$ be defined by $\sup_x \varphi(x,y) > 0$.
  If $g = \tp(a,b) \in U$, then clearly $\tp(b) \in V$.
  Conversely, if $\tp(b) \in V$, and $M$ is the model coded by $b$, then $b \in D(M)$, so there exists $a \in D(M)$ such that $\varphi(a,b) > 0$.
  By \autoref{lem:UniversalSkolemSortModel}, there exists $a' \in D(M)$ arbitrarily close to $a$ such that $\dcl(a') = M$, i.e., $g = \tp(a,b) \in \bG(T)$.
  Taking $d(a',a)$ small enough we have $\varphi(a,b) > 0$, i.e., $g \in U$.
  In either case, $s_g = \tp(b)$, so $V = s(U)$ and $\bG(T)$ is open.
\end{proof}

When $T$ is classical, the sort $D_\Phi$ is universal Skolem by \autoref{prp:UniversalSkolemSortClassical}, so our construction generalises that of \autoref{sec:GroupoidTheoryClassical}.
When $T$ is $\aleph_0$-categorical, if $G(T) = \Aut(M)$ for any separable $M \vDash T$, then $\bG(T) = 2^\bN \times G(T) \times 2^\bN$, by \autoref{prp:UniversalSkolemSortAleph0Categorical}, generalising \autoref{prp:ClassicalGTAleph0Categorical}.


We turn to the reconstruction of $T$, up to bi-interpretation, from the topological groupoid $\bG = \bG(T)$, relative to some fixed universal Skolem sort $D$.
We shall attempt to keep this as close as possible to what was done in \autoref{sec:ClassicalReconstruction}, despite some unavoidable differences.
Our precise aim is to recover the theory $T^D$, in the single sort $D$ (and not in all the interpretable sorts, of which there are uncountably many).
Similarly, aiming to recover a metric sort (rather than discrete ones), the role of clopen sub-groupoids will be taken over by compatible (semi-)norms.

\begin{dfn}
  \label{dfn:NeighbourhoodOfSet}
  Let $X$ be a topological space.
  By a \emph{neighbourhood} of a (usually compact) subset $K \subseteq X$ we mean any set containing an open set containing $K$.
  A \emph{basis of neighbourhoods} for $K$ is a family of neighbourhoods that is cofinal among all neighbourhoods with respect to inverse inclusion.
\end{dfn}

\begin{dfn}
  \label{dfn:Norm}
  A \emph{semi-norm} on a groupoid $\bG$ is a function $\rho \colon \bG \rightarrow \bR^+$ which vanishes on $\bB$ and satisfies
  \begin{gather*}
    \rho(g) = \rho(g^{-1}), \qquad \rho(fg) \leq \rho(f) + \rho(g) \qquad \text{when $fg$ is defined}.
  \end{gather*}
  It is a \emph{norm} if it vanishes only on $\bB$, and it is \emph{compatible} (with the topology) if it continuous and the sets $\{\rho < r\} = \bigl\{ g \in \bG : \rho(g) < r \bigr\}$ form a basis of neighbourhoods for $\bB$.
\end{dfn}

Clearly, any two compatible norms $\rho$ and $\rho'$ must be \emph{uniformly equivalent}: for every $\varepsilon > 0$ there exists $\delta > 0$ such that $\{\rho < \delta\} \subseteq \{\rho' < \varepsilon\}$ and vice versa.
However, a compatible norm on a topological groupoid does not suffice to recover the topology (while it does for a topological group), and assuming that $\{\rho < r\}$ forms a basis of neighbourhoods for $\bB$ does not imply that $\rho$ is continuous.
For our purposes it will suffice to keep in mind the analogy with \autoref{sec:ClassicalReconstruction}: a $0/1$-valued continuous semi-norm is the same thing as the \emph{$0$-characteristic function} of a clopen sub-groupoid $\bH \leq \bG$ that contains $\bB$ (i.e., $\rho(g) = 0$ if $g \in \bH$ and $\rho(g) = 1$ otherwise).

Let us start by considering formulas in two variables (all in $D$), since the generalisation to more variables is straightforward.
Such a formula $\varphi(x,y)$ defines a continuous bounded function that will also be denoted $\varphi \colon \tS_{2 D}(T) \rightarrow \bR$.
Its restriction to $\bG$ will be denoted $\varphi_\bG$.
We may also write
\begin{gather*}
  [\varphi < r] = \bigl\{ p \in \tS_{2D}(T) : \varphi(p) < r \bigr\},
  \qquad
  [\varphi < r]_\bG = [\varphi < r] \cap \bG.
\end{gather*}

Given any two bounded functions $\xi,\zeta\colon \bG \rightarrow \bR$, let us define
\begin{gather*}
  (\xi * \zeta)(f) = \inf \, \bigl\{ \xi(g) + \zeta(h) : f = gh \bigr\}.
\end{gather*}
In particular, any semi-norm satisfies $\rho * \rho = \rho$.
The analogue of \autoref{lem:ClassicalReconstructionProductOfFormulas} is:

\begin{lem}
  \label{lem:ContinuousReconstructionProductOfFormulas}
  Let $\varphi(x,y)$ and $\psi(y,z)$ be formulas with variables in $D$, and let $\chi(x,z)$ be the formula $\inf_y (\varphi + \psi)$.
  Then
  \begin{gather*}
    \varphi_\bG * \psi_\bG = \chi_\bG.
  \end{gather*}
\end{lem}
\begin{proof}
  The inequality $\geq$ is clear.
  For the opposite inequality assume that $f = \tp(a,c) \in \bG$ and $\chi_\bG(f) = \chi(a,c) < r$.
  Then $a$ and $c$ both code the same separable model $M$, and there exists $b \in D(M)$ such that $\varphi(a,b) + \psi(b,c) < r$.
  By \autoref{lem:UniversalSkolemSortModel}, we can find $b' \in D(M)$ that also codes $M$ arbitrarily close to $b$.
  This means that $g = \tp(a,b') \in \bG$ and $h = \tp(b',c) \in \bG$.
  Since formulas are always uniformly continuous, we may choose $b'$ close enough to $b'$ that $\varphi_\bG(g) + \psi_\bG(h) = \varphi(a,b') + \psi(b',c) < r$.
  In addition, $f = gh$, so $(\varphi_\bG * \psi_\bG)(f) < r$ as well.
\end{proof}

It follows that if $d$ is any definable distance on $D$ (and we might as well fix one now), then $d_\bG$ is a continuous norm on $\bG$.
The following is the analogue of \autoref{lem:ClassicalReconstructionSubGroupoidContainsEn}:

\begin{lem}
  \label{lem:ContinuousReconstructionCompatibleNorm}
  If $d$ is a definable distance on $D$, then $d_\bG$ is a compatible norm on $\bG$.
\end{lem}
\begin{proof}
  We still need to show that every neighbourhood $U$ of $\bB$ contains a set of the form $\{d_\bG < r\}$.
  If $e \in \bB$, then $U$ contains a basic neighbourhood of $e$, namely of the form $[\varphi < 1]_\bG = \bigl\{g \in \bG : \varphi(g) < 1 \bigr\}$ for some formula $\varphi(x,y)$ that vanishes at $e$.
  Since $\varphi$ is uniformly continuous, for $r > 0$ small enough we have $e \in [\varphi(x,x) < 1/2]_\bG \cap [d(x,y) < r]_\bG$.
  From this point we proceed as in the proof of \autoref{lem:ClassicalReconstructionSubGroupoidContainsEn}, using compactness of $\bB$ to find a finite cover $\bB \subseteq \bigcup_{i<k} [\varphi(x,x) < 1/2]_\bG$ and $r>0$ that works for all $e_i$, so $[d < r]_\bG = \{d_\bG < r\} \subseteq U$.
\end{proof}

The analogy of the next steps is somewhat less clear: we work exclusively within the sort $D$, so the projection $\pi$ of \autoref{sec:ClassicalReconstruction} has no analogue.
Still, in some twisted way, the following is at least related to \autoref{lem:ClassicalReconstructionProjections}.

\begin{lem}
  \label{lem:ContinuousReconstructionOpenThickening}
  Let $U \subseteq \bG$ be open, $d$ be a definable distance on $D$, and $\delta > 0$.
  Define $V = (U)_{d<\delta} \subseteq \tS_{2D}(T)$ to be the set of all $p = \tp(a,b) \in \tS_{2 D}(T)$ for which there exists $g = \tp(c,d) \in U$ with $d(a,c) \vee d(b,d) < \delta$.
  Then $V$ is open in $\tS_{2 D}(T)$.
\end{lem}
\begin{proof}
  Indeed, let $p \in V$, as witnessed by $g \in U$.
  Since $U$ is open, there exists a basic open set $U_0 = [\varphi < 1]_\bG$ such that $h \in U_0 \subseteq U$.
  Let
  \begin{gather*}
    \chi(x,y) = \inf_{u,v} \left[ \varphi(u,v) \vee \frac{d(x,u)}{\delta} \vee \frac{d(y,v)}{\delta} \right].
  \end{gather*}
  Clearly, $p \in [\chi < 1]$.

  Assume now that $\chi(a',b') < 1$.
  This is witnessed by some $c',d'$ such that $\varphi(c',d') < 1$ and $d(a',c') \vee d(a',d') < \delta$.
  Since every formula is uniformly continuous, this remains true if we move $c'$ and $d'$ by a sufficiently small amount.
  In particular, by \autoref{lem:UniversalSkolemSortModel}, we may assume that $c'$ and $d'$ that both code some model $M$.
  Then $\tp(c',d') \in U$, and it witnesses that $\tp(a',b') \in V$.

  We have thus shown that $p \in [\chi < 1] \subseteq V$, so $V$ is indeed open.
\end{proof}

At any rate, the following is analogous to \autoref{lem:ClassicalReconstructionFormulaClopenEnInvariant}.

\begin{dfn}
  \label{dfn:ContinuousReconstructionUCC}
  Say that a continuous function $\xi\colon \bG \rightarrow \bR$ is \emph{uniformly continuous and continuous}, or \emph{UCC}, if it is continuous, and in addition, for every $\varepsilon > 0$ there exists a neighbourhood $U$ of $\bB$ such that $|\xi(g) - \xi(h)| < \varepsilon$ whenever $h \in U g U$.
\end{dfn}

\begin{rmk}
  \label{rmk:ContinuousReconstructionUCC}
  If $\rho$ is any compatible norm, then $\xi$ is UCC if and only if it is continuous, and for every $\varepsilon > 0$ there exists $\delta > 0$ such that $|\xi(g) - \xi(fgh)| < \varepsilon$ whenever $fgh$ is defined and $\rho(f) \vee \rho(h) < \delta$.
\end{rmk}

\begin{lem}
  \label{lem:ContinuousReconstructionUCC}
  If $\varphi(x,y)$ is a formula, then $\varphi_\bG\colon \bG \rightarrow \bR$ is UCC, and conversely, every UCC function on $\bG$ is of this form, for a unique formula $\varphi$.
\end{lem}
\begin{proof}
  The first assertion follows from standard facts: every formula is a uniformly continuous function of its arguments, and a continuous function of their types.
  For the converse, it will suffice to prove that a UCC function $\xi\colon \bG \rightarrow \bR$ extends to a (necessarily unique) continuous function on $\tS_{2 D}(T)$.

  For this, let $p = \tp(a,b) \in \tS_{2 D}(T)$ and $\varepsilon > 0$ be given.
  Let $d$ be a definable distance on $D$, and fix $\delta > 0$ as in \autoref{rmk:ContinuousReconstructionUCC}, for $\rho = d_\bG$.
  We may choose a separable model $M$ that contains both $a$ and $b$.
  By \autoref{lem:UniversalSkolemSortModel} we may choose $c$ and $d$ that code $M$, and in addition $d(a,c) \vee d(b,d) < \delta$.
  In particular, $g = \tp(c,d)$ belongs to $\bG$.

  Without loss of generality we may assume that $\xi(g) = 0$, and let $U = \{|\xi| < \varepsilon\}$, an open subset of $\bG$.
  Let $V = (U)_{d<\delta} \subseteq \tS_{2 D}(T)$ as in \autoref{lem:ContinuousReconstructionOpenThickening}.
  Then $V$ is open in $\tS_{2 D}(T)$ and $p \in V$ by construction.
  In order to finish the proof, it will suffice to show that $|\xi| \leq 2\varepsilon$ on $V \cap \bG$.

  So let $h' = \tp(a',b') \in V \cap \bG$, and assume toward a contradiction that $\xi(h') > 2\varepsilon$.
  Let $g' \in U$ witness that $h' \in V$, so $g' = \tp(c',d')$ and $d(a',c') \vee d(b',d') < \delta$.

  We can find a basic open set $g' \in U_0 = [\psi < 1]_\bG \subseteq U$.
  Since $\psi$ is uniformly continuous, if $g'' = \tp(c'',d'') \in \bG$ and $c''$ and $d''$ are close enough to $c'$ and $d'$, then
  \begin{gather*}
    \bigl|\psi(c',d') - \psi(c'',d'') \bigr| < 1 - |\psi(c',d')|,
  \end{gather*}
  so $g'' \in U_0 \subseteq U$ as well.
  A similar consideration applies for $h' \in W = \{\xi > 2\varepsilon\}$.
  We may now apply \autoref{lem:UniversalSkolemSortModel} to find $a''$, $b''$, $c''$ and $d''$ that code a common model $M$ and are sufficiently close to $a'$, $b'$, $c'$ and $d'$, respectively, that
  \begin{gather*}
    g'' = \tp(c'',d'') \in U,
    \qquad
    h'' = \tp(a'',b'') \in W,
    \qquad
    d(a'',c'') \vee d(b'',d'') < \delta.
  \end{gather*}
  But now
  \begin{gather*}
    g'' = \tp(c'',a'') \cdot h'' \cdot \tp(b'',d''),
  \end{gather*}
  so $|\xi(h'') - \xi(g'')| < \varepsilon$ by choice of $\delta$, a contradiction.

  To sum up, for every $p \in \tS_{2 D}(T)$ and $\varepsilon > 0$ we found an open neighbourhood $V$ of $p$ such that $\xi$ varies by no more than $4\varepsilon$ on $V \cap \bG$.
  It follows that $\xi$ can be extended to a continuous function on $\tS_{2 D}(T)$, i.e., to a formula.
\end{proof}

The following is clearly analogous to \autoref{lem:ClassicalReconstructionSorts}.
If $\rho$ is a (semi-)norm and $f,g \in \bG$ have the same target, let $d_L^\rho(f,g) = \rho(f^{-1},g)$, which defines a (pseudo-)distance on $e \bG$ for each $e \in \bB$ (the $L$ stands for \emph{left-invariant}: $d_L^\rho(f,g) = d_L^\rho(hf,hg)$ whenever $t_f = t_g = s_h$).

\begin{lem}
  \label{lem:ContinuousReconstructionDistance}
  The map $d \mapsto d_\bG$ defines a bijection between definable distances on $D$ and compatible norms on $\bG$.

  In addition, let $d$ be such a distance, let $a \in D$ code $M$ and $e = \tp(a)$, and let $D_0 \subseteq D(M)$ be the set of $b \in D(M)$ that code $M$.
  Then $\tp(a,b) \mapsto b$ is an isometric bijection of $(e \bG,d_L^\rho)$ with $(D_0,d)$, that extends to an isometric bijection
  \begin{gather*}
    \widehat{(e\bG,d_L^\rho)} \cong (D,d).
  \end{gather*}
\end{lem}
\begin{proof}
  We have already observed in \autoref{lem:ContinuousReconstructionCompatibleNorm} that if $d$ is a definable distance on $D$, then $d_\bG$ is a compatible norm.
  For the converse, let $\rho$ be any compatible norm.
  Then it is UCC (by \autoref{rmk:ContinuousReconstructionUCC}), so $\rho = d_\bG$ for some formula $d(x,y)$, which is necessarily the unique continuous extension of $\rho$ to $\tS_{2 D}(T)$.

  We have $d(x,x) = 0$ since $\rho$ vanishes on $\bB$, and $d(x,y) = d(y,x)$ since $\rho(g) = \rho(g^{-1})$.
  In addition, we have $\rho * \rho = \rho$, so by uniqueness of the reconstructed formula and \autoref{lem:ContinuousReconstructionProductOfFormulas},
  \begin{gather*}
    d(x,z) = \inf_y \, \bigl( d(x,y) + d(y,z) \bigr).
  \end{gather*}
  Therefore $d$ defines a distance, and $\rho = d_\bG$.

  The second part is essentially tautological.
\end{proof}

As in \autoref{sec:ClassicalReconstruction}, in order to recover formulas in several variables, we need to replace $\bG \cong \bG^{[2]}$ with $\bG^{[k]} = \bG \backslash \bG^{k/t}$ for arbitrary $k \geq 1$, which we identify with the set of $\tp(a_i : i <k) \in \tS_{k D}(T)$ for which all the $a_i \in D$ code the same model.
In particular, $\bG^{[k]} \subseteq \tS_{k D}(T)$, and is even dense there.
If $\varphi(x_i : i < k)$ is a formula, then we identify it with the corresponding continuous function $\varphi\colon \tS_{k D}(T) \rightarrow \bR$, and let $\varphi_\bG$ be its restriction to $\bG^{[k]}$.
For $U \subseteq \bG^{[k]}$ we may define $(U)_{d<\delta} \subseteq \tS_{k D}(T)$ to consist of all $\tp(b_i : i < k)$ such that there exists $\tp(a_i : i < k) \in U$ satisfying $d(a_i,b_i) < \delta$ for all $i$.
We say that $\xi: \bG^{[k]} \rightarrow \bR$ is UCC if it is continuous, and for every $\varepsilon > 0$ there exists a neighbourhood $\bB \subseteq U$ such that, if $q \in p \cdot U^k$ (with respect to the action $\bG^{[k]} \curvearrowleft \bG^k$), then $|\xi(p) - \xi(q)| < \varepsilon$.

\begin{lem}
  \label{lem:ContinuousReconstructionGk}
  Let $k \geq 1$ and let $d$ be a definable distance on $D$.
  \begin{enumerate}
  \item If $U \subseteq \bG^{[k]}$ is open, then $(U)_{d<\delta}$ is open in $\tS_{k D}(T)$.
  \item \label{item:ContinuousReconstructionGkModulus}
    Let $\rho$ be a compatible norm on $\bG$.
    A continuous function $\xi\colon \bG^{[k]} \rightarrow \bR$ is UCC if and only if for every $\varepsilon > 0$ there exists $\delta > 0$ such that, if if $q \in p \cdot \{\rho < \delta\}^k$, then $|\xi(p) - \xi(q)| < \varepsilon$.
  \item If $\varphi(x_i : i < k)$ is a formula, then $\varphi_\bG\colon \bG^{[k]} \rightarrow \bR$ is UCC, and conversely, every UCC function on $\bG^{[k]}$ is of this form, for a unique formula $\varphi$.
  \end{enumerate}
\end{lem}
\begin{proof}
  As for \autoref{lem:ContinuousReconstructionOpenThickening} and \autoref{lem:ContinuousReconstructionUCC}, \textit{mutatis mutandis}.
\end{proof}

\begin{thm}
  Let $T$ and $T'$ be two complete theories with universal Skolem sorts.
  Then $\bG(T) \cong \bG(T')$ as topological groupoids if and only if $T$ and $T'$ are bi-interpretable.
  Moreover, given $\bG = \bG(T)$ as a topological groupoid, we can reconstruct the theory $T^D$, up to a change of language and choice of distance on $D$ (among all definable distances).
\end{thm}
\begin{proof}
  For the first assertion, one direction has already been observed ($\bG(T)$ only depends on $T$ up to bi-interpretation), and the other direction follows from the moreover part.

  For the reconstruction, we must first choose (arbitrarily) a compatible norm $\rho$ on $\bG$, which, by \autoref{lem:ContinuousReconstructionCompatibleNorm}, is the same thing as choosing a definable distance $d$ on the universal Skolem sort $D$ (so $\rho = d_\bG$).
  We define $\cL^D$ to consist of a single metric sort, together with a $k$-ary predicate symbol $P_\xi$ for each UCC function $\xi$ on $\bG^{[k]}$ (or for a countable uniformly dense family of such functions).

  We need to specify a bound and a continuity modulus for each symbol: if $\xi$ is UCC, then \autoref{lem:ContinuousReconstructionGk}\autoref{item:ContinuousReconstructionGkModulus} (with the chosen $\rho$) provides us with a modulus of continuity.
  In addition, $P_\xi$ is the restriction of a formula, and therefore bounded.
  In particular, we use the bound on $\rho$ for a bound on the diameter in $\cL^D$.

  Next, for each $e \in \bB$ we define $M_e = \widehat{(e \bG,d_L^\rho)}$ (where we recall that $d_L^\rho(f,g) = \rho(f^{-1}g)$).
  We interpret each predicate $P_\xi$ on $e \bG$ as:
  \begin{gather*}
    P_\xi(g) = \xi(\bG g).
  \end{gather*}
  It satisfies the prescribed bound and continuity modulus, and in particular extends continuously to all of $M_e$.

  If $e = \tp(a)$, where $a$ codes $M$, and if we identify $P_\xi$ with the formula $\varphi$ of $T$ such that $\xi = \varphi_\bG$, then $M_e$ is isomorphic to $D(M)$.
  Then the theory of any $M_e$ (or of the family of all of them) is, up to a change of language, $T^D$.
\end{proof}

\section{Further questions}
\label{sec:Questions}

We have intentionally kept this paper relatively short, with the bare minimum of associating $\bG(T)$ to $T$ and reconstructing $T$ from $\bG(T)$.
Let us point out some further topics for research.
About some of them some progress has already made, and they may be treated in a subsequent paper, whiles others are wide open.

\subsection*{General groupoids}

It follows from our work that $\bG = \bG(T)$ admits compatible norms, and that UCC functions on $\bG$, or even on $\bG^{[k]}$, separate points and closed sets (i.e., determine the topology).
For a general topological groupoid, even assuming that it is open, (completely) metrisable and that $\bB$ is the Cantor space, the best we can show is that it admits a \emph{semi-compatible} norm, namely one such that the sets $\{\rho < r\}$ for a basis of open neighbourhoods of identity, so it is upper semi-continuous, but not necessarily continuous.
We can also show that under reasonable hypotheses, the existence of a compatible norm and of sufficiently many UCC functions are equivalent (this will appear in a subsequent paper).

In groups, UC functions (uniformly continuous with respect to the Roelcke uniformity) are analogous to our UCC functions, and are closely related to the Roelcke completion and the Roelcke compactification.
A polish group $G$ is of the form $G(T)$, for $\aleph_0$-categorical $T$, if and only if it is Roelcke pre-compact (see \cite{BenYaacov-Tsankov:WAP}), i.e., if and only if the compactification and completion agree.

\begin{qst}
  State general hypotheses under which a topological groupoid admits a compatible norm / sufficiently many UCC functions.
  The conditions must hold when $\bG = \bG(T)$, and not refer to the $\bG(T)$ construction explicitly.
\end{qst}

\begin{qst}
  Construct analogues of the Roelcke compactification and the Roelcke completion of a groupoid (possibly under certain hypotheses).
  When $\bG = \bG(T)$, both should be $\tS_{2 D}(T)$.
\end{qst}

\begin{qst}
  Characterise topological groupoids of the form $\bG(T)$.
  Ideally, the characterisation should be: first, some general conditions hold, ensuring in particular that the Roelcke completion makes sense, and second, the Roelcke completion is compact.
\end{qst}

\subsection*{Universal Skolem sorts and possible generalisations}

We have only constructed $\bG(T)$ when $T$ admits a universal Skolem sort.
We know that this is true when $T$ is classical or $\aleph_0$-categorical.
In his Ph.D.\ dissertation (in progress), Jorge Muñoz shows that if $T$ admits a universal Skolem sort, then so does its randomisation $T^R$ (see \cite{BenYaacov-Keisler:MetricRandom,BenYaacov:RandomVariables}), giving an explicit construction of one sort from the other.
On the other hand, in \autoref{exm:NoSkolemSort} we showed that a Skolem sort need not always exist.

\begin{qst}
  Can the $\bG(T)$ construction be extended, or generalised, to all theories?
\end{qst}

By \emph{generalise} we mean something similar to how our groupoid construction and reconstruction relate to the $\aleph_0$-categorical situation: the groupoid $\bG(T)$ is not the same as the group $G(T)$, but one can be trivially recovered from the other.

\subsection*{The category of interpretations}

We have shown that isomorphisms of $\bG(T)$ and $\bG(T')$ correspond to bi-interpretations of $T$ and $T'$.
When $T$ and $T'$ are $\aleph_0$-categorical, interpretations on $T'$ in $T$ correspond to continuous morphisms $G(T) \rightarrow G(T')$ such that the isometric action $G(T) \curvearrowright \widehat{G(T')}_L$ has compactly many orbit closures (see \cite{BenYaacov-Kaichouh:Reconstruction}).
More precisely, the category of interpretations of $\aleph_0$-categorical theories (modulo a reasonable equivalence relation) is equivalent to the category of Roelcke pre-compact Polish groups with such morphisms.

\begin{qst}
  Provide a correspondence between interpretations of $T'$ in $T$ and (special) morphisms of groupoids $\bG(T) \rightarrow \bG(T')$.
  The category of interpretations of complete countable theories should be equivalent to the category of $\bG(T)$ with some condition on the morphisms.
\end{qst}

\subsection*{Model theoretic properties}

One of the motivations for the present work lies in the fact that for an $\aleph_0$-categorical theory $T$, model-theoretic properties (in particular, stability and NIP) correspond to dynamical properties of the system $G \curvearrowright R$, where $G = G(T)$ and $R$ is its Roelcke completion/compactification (see \cite{BenYaacov-Tsankov:WAP,Ibarlucia:DynamicalHierarchy}).

\begin{qst}
  Extend the above to the case where $\bG = \bG(T)$, so the corresponding dynamical system should be $\bG(T) \curvearrowright \tS_{2 D}(T)$.
\end{qst}

\begin{qst}
  Together with work of Muñoz alluded to above, extend the preservation arguments of \cite{Ibarlucia:AutRand} from $\aleph_0$-categorical theories to arbitrary ones (admitting a universal Skolem sort).
\end{qst}

\bibliographystyle{begnac}
\bibliography{begnac}

\end{document}